\documentclass[12pt,reqno,twoside]{article}
\usepackage[left=1 in,top=1 in,right=1 in,bottom=1 in]{geometry}
\usepackage{amssymb,amsfonts,amsthm,amsmath}
\usepackage{color}
\usepackage{hyperref}
\usepackage{cite}
\usepackage{enumitem}

\numberwithin{equation}{section}

\newcommand{\ddt}{\frac{d}{dt}}
\newcommand{\R}{\mathbb R}
\newcommand{\N}{\mathbb N}

\def\eqdef{\stackrel{\rm def}{=}}

\def\beq{\begin{equation}}
\def\eeq{\end{equation}}

\def\beqs{\begin{equation*}}
\def\eeqs{\end{equation*}}

\newtheorem{theorem}{Theorem}[section]
\newtheorem{lemma}[theorem]{Lemma}
\newtheorem{proposition}[theorem]{Proposition}
\newtheorem{corollary}[theorem]{Corollary}

\theoremstyle{definition}
\newtheorem{remark}[theorem]{Remark}

\newtheorem{definition}[theorem]{Definition}

\newtheorem*{xnotation}{Notation}

\definecolor{darkred}{rgb}{.70,.12,.20}

\definecolor{darkgreen}{rgb}{.20,.52,.14}

\def\cprime{$'$}

\newcommand{\supp}{\operatorname{supp}}
\newcommand{\varep}{\varepsilon}

\newcommand{\eR}{\vec k}
\newcommand{\Jm}{{\mathbf J}}
\newcommand{\Ncal}{{\mathcal N}}
\newcommand{\Ecal}{{\mathcal E}}
\newcommand{\hh}{\Phi}
\newcommand{\KK}{\mathcal K}
\newcommand{\ZZ}{W}
\newcommand{\MZ}{M_{\mathcal Z}}
\newcommand{\muz}{\mu_{\mathcal Z}}

\newcommand{\inttx}[3]{\int_{#1}^{#2}\!\!\!\int_{#3}}
\newcommand{\intTU}{\inttx{0}{T}{U}}

\newcommand{\intTUk}{\inttx{0}{T}{U_k}}
\newcommand{\intTV}{\inttx{0}{T}{V}}
\newcommand{\intTUp}{\inttx{0}{T}{U'}}

\numberwithin{equation}{section}

\title{\textbf{Slightly Compressible Forchheimer Flows in Rotating Porous Media}}
\author{Emine Celik$^a$, Luan Hoang$^b$,  and Thinh Kieu$^c$}

\date{\today}

\begin{document}

\maketitle

\begin{center}
\textit{$^a$Department of Mathematics, Sakarya University\\
54050, Sakarya, Turkey}\\
\textit{$^b$Department of Mathematics and Statistics, Texas Tech University\\
Box 41042, Lubbock, TX 79409--1042, U. S. A.} \\
\textit{$^c$Department of Mathematics, University of North Georgia, Gainesville Campus\\
3820 Mundy Mill Rd., Oakwood, GA 30566, U. S. A.}\\
Email addresses: \texttt{eminecelik@sakarya.edu.tr, luan.hoang@ttu.edu, thinh.kieu@ung.edu}
\end{center}

\begin{abstract}
We formulate the the generalized Forchheimer equations for the three-dimensional fluid flows in rotating porous media. 
By implicitly solving the momentum in terms of the pressure's gradient, we derive a degenerate parabolic equation for the density in the case of  slightly compressible fluids and study its corresponding initial, boundary value problem.
We investigate the nonlinear structure of the parabolic equation.
The maximum principle is proved and used to obtain the maximum estimates for the solution.
Various estimates are established for the solution's gradient, in the Lebesgue norms of any order, in terms of the initial and boundary data. All estimates contain explicit dependence on key physical parameters including the angular speed.
\end{abstract}



\tableofcontents

\pagestyle{myheadings}\markboth{E. Celik, L. Hoang,  and T. Kieu}
{Forchheimer flows in rotating porous media}

\section{Introduction and formulation of the problem}\label{intro}

We study slightly compressible fluid flows in rotating porous media in the three dimensional space.
The fluid has density $\rho$, velocity $v$, pressure $p$, and dynamic viscosity $\mu$. 
The porous medium has constant porosity $\tilde \phi\in(0,1)$ and permeability $k$.
It is rotated with a constant angular velocity $\vec\Omega$, and an associated rotating frame is given. 
In this rotating coordinate system, $\vec\Omega$ is written as $\tilde\Omega \eR$, where  $\tilde\Omega\ge 0$ is the constant angular speed, and $\eR$ is a constant unit vector. Let $x$ be the coordinate vector of a position in this rotating frame.

The equation for fluid flows written in the rotating frame, see e.g. \cite{VadaszBook}, is
\beq\label{Drot}
\frac{\mu}{k}v+ \frac{2\rho \tilde\Omega }{\tilde\phi} \eR\times v + \rho \tilde\Omega^2 \eR\times (\eR\times x)=-\nabla p+\rho \vec g,
\eeq
where 
 $(2\rho\tilde\Omega/\tilde\phi) \eR\times v$ represents the Coriolis effect in porous media,
$\tilde\Omega^2 \eR\times (\eR\times x)$ is the centripetal acceleration,
and $\vec g$ is the gravitational acceleration. 

The basic assumption for equation \eqref{Drot} is that the flows obey the Darcy's law
\beq\label{Deq}
\frac{\mu}{k}v=-\nabla p+\rho \vec g.
\eeq

However, in many situations, for instance, when the Reynolds number is large, this assumption is invalid. Instead, Forchheimer equations \cite{Forchh1901,Forchheimerbook} are usually used to model the flows in these cases. For example, the two-term Forchheimer's law states that
\beq\label{Forch2law}
av+b|v|v=-\nabla +\rho \vec g,
\eeq
where  $a, b>0$ are some physical parameters.
(See also Forchheimer's three-term and power laws in, e.g., \cite{MuskatBook,BearBook,NieldBook}.)

A general form of the Forchheimer equations, which extends \eqref{Deq} and \eqref{Forch2law} taking into account Muskat's dimension analysis \cite{MuskatBook}, is
\beq\label{Fnr}
 \sum_{i=0}^N a_i \rho^{\alpha_i} |v|^{\alpha_i} v=-\nabla p +\rho \vec g.
\eeq
Here we focus on the explicit dependence on the density, leaving the dependence on the dynamic viscosity and permeability encoded in the coefficients $a_i$'s. 

The interested reader is referred to the books \cite{BearBook,NieldBook,StraughanBook} for more information about the Forchheimer equations and a larger family of Brinkman-Darcy-Forchheimer equations. For their mathematical analysis in the case of incompressible fluids, see e.g. \cite{Zabensky2015a,ChadamQin,Payne1999a,Payne1999b,MTT2016,CKU2006,KR2017} and references therein. For the treatments of compressible fluids, see \cite{ABHI1,HI1,HI2,HIKS1,HKP1,HK1,HK2,CHK1,CHK2,CH1,CH2}. It is noted that the Forchheimer flows have drawn much less attention of mathematical research compared to the Darcy flows, and among the papers devoted to them, the number of those on compressible fluids is much smaller than that on the incompressible one. 

Then the  equation for the rotating flows corresponding to \eqref{Fnr}, written in the rotating frame, is
 \beq\label{FM}
 \sum_{i=0}^N a_i \rho^{\alpha_i} |v|^{\alpha_i} v+ \frac{2\rho \tilde\Omega}{\tilde\phi} \eR\times v+\rho \tilde\Omega^2 \eR\times (\eR\times x)=-\nabla p +\rho \vec g.
 \eeq

In particular, when $N=1$, the specific Forchheimer's two-term law for rotating fluids \cite{Ward64,VadaszBook} is 
 \beqs
\frac{\mu}{k} v+ \frac{c_F \rho}{\sqrt{k}}|v|v + \frac{2\rho \tilde\Omega}{\tilde\phi} \eR\times v+\rho \tilde\Omega^2 \eR\times (\eR\times x)=-\nabla p +\rho \vec g,
 \eeqs
where $c_F$ is the Forchheimer constant.
Even in this case, there is no mathematical analysis for compressible fluids in literature.

We make one simplification  in \eqref{FM}: replacing $\displaystyle \frac{2\rho \tilde\Omega}{\tilde\phi} \eR\times v$  with  
\beq\label{sim}
 \frac{2\rho_* \tilde\Omega}{\tilde\phi} \eR\times v, \quad \rho_*=const.\ge 0.
\eeq

We then approximate equation \eqref{FM} by  
 \beq\label{FMR}
 \sum_{i=0}^N a_i \rho^{\alpha_i} |v|^{\alpha_i} v+ \mathcal R \eR\times v+\rho \tilde\Omega^2 \eR\times (\eR\times x)=-\nabla p +\rho \vec g,
 \eeq
where
\beq\label{Rdef}
\mathcal R=\frac{2\rho_* \tilde\Omega}{\tilde\phi}=const.\ge 0.
\eeq

Let $g:\mathbb{R}^+\rightarrow\mathbb{R}^+$ be a generalized polynomial  defined by
\beq\label{eq2}
g(s)=a_0 + a_1s^{\alpha_1}+\cdots +a_Ns^{\alpha_N}=\sum_{i=0}^N a_i s^{\alpha_i}\quad\text{for } s\ge 0,
\eeq 
where $N\ge 1$ is an integer, the powers $\alpha_0=0<\alpha_1<\alpha_2<\ldots<\alpha_N$ are real numbers, 
and the coefficients $a_0,a_1,\ldots,a_N$ are positive constants.

Then equation \eqref{FMR} can be rewritten as 
\beq\label{FM1}
 g(\rho|v|) v+ \mathcal R \eR\times v=-\nabla p +\rho \vec g - \rho \tilde\Omega^2 \eR\times (\eR\times x).
\eeq

Multiplying both sides of  \eqref{FM1}  by $\rho$ gives
 \beq\label{neweq1}
 g(|\rho v|) \rho v + \mathcal R \eR\times (\rho v)  =-\rho\nabla p+ \rho^2 \vec g- \rho^2 \tilde\Omega^2 \eR\times (\eR\times x).
 \eeq 

We will solve  for $\rho v$ from \eqref{neweq1}, which is possible thanks to the following lemma.

\begin{lemma}\label{Fsolve}
Given any vector $k\in\R^3$, the function $F_0(v)\eqdef g(|v|)v + k\times v$ is a bijection from $\R^3$ to $\R^3$.
\end{lemma}
\begin{proof}
Note that $F_0$ is a continuous function on $\R^3$ and 
$$ \frac{F_0(v)\cdot v}{|v|}=\frac{g(|v|)|v|^2}{|v|}\to\infty \text{ as }|v|\to\infty.$$
Then it is well-known that $F_0(\R^3)=\R^3$, see e.g. \cite[Theorem 3.3]{Deimling2010}.

It remains to prove that $F_0$ is one-to-one. 
Let $v, w\in \R^3$. We have
\beqs
\begin{split}
(F_0(v)-F_0(w))\cdot (v-w)
&= (g(|v|)v-g(|w|))\cdot (v-w)+ [k\times (v-w)]\cdot(v-w)\\ 
&=(g(|v|)v-g(|w|)w)\cdot (v-w) \\
&= a_0|v-w|^2+\sum_{i=1}^N a_i (|v|^{\alpha_i}v-|w|^{\alpha_i} w)\cdot (v-w).
\end{split}
\eeqs

By applying \cite[Lemma 4.4, p.~13]{DiDegenerateBook} to each 
$(|v|^{\alpha_i}v-|w|^{\alpha_i} w)\cdot (v-w)$, for $i\ge 1$,
we then obtain
\beq\label{Fmono}
(F_0(v)-F_0(w))\cdot (v-w)\ge a_0|v-w|^2+\sum_{i=1}^N C_i a_i |v - w|^{\alpha_i+2},   
\eeq
where $C_i>0$ depends on $\alpha_i$ and $a_i$, for $i=1,\ldots, N.$ 

If $F_0(v)=F_0(w)$, it follows the monotonicity \eqref{Fmono} that 
\beqs
0=(F_0(v)-F_0(w))\cdot (v-w)\ge a_0 |v - w|^2,   
\eeqs
which implies $v=w$.
\end{proof}

Let vector $\eR=(k_1,k_2,k_3)$  be fixed now with $k_1^2+k_2^2+k_3^2=1$.
We denote by $\Jm$ the $3\times 3$ matrix for which $\Jm x=\eR\times x$ for all $x\in\R^3$. Explicitly, 
we have
\beq\label{JR}
\Jm =\begin{pmatrix}   0 & -k_3 & k_2\\ k_3 & 0 & -k_1 \\ -k_2& k_1 & 0 \end{pmatrix}
\text{ and }
\Jm^2=\begin{pmatrix}
     k_1^2-1&k_1 k_2 & k_1 k_3\\
     k_1 k_2&k_2^2-1&k_2 k_3\\
     k_1 k_3& k_2 k_3&k_3^2-1
    \end{pmatrix}.
\eeq

\begin{definition}\label{FX} Throughout the paper, the function $g$ in \eqref{eq2} is fixed. We define the function $F:\R^3\to\R^3$ by
\beq\label{Fdef}
F(v)= g(|v|)v + \mathcal R  \Jm v \quad \text{ for }v\in\R^3,
\eeq
and denote its inverse function, which exists thanks to Lemma \ref{Fsolve}, by 
\beq \label{Xdef}
X=F^{-1}.
\eeq
\end{definition}

Since $F$ is odd, then so is $X$. 
Returning to equation \eqref{neweq1}, we can invert 
 \beq\label{new2}
\rho v= -X(\rho\nabla p - \rho^2 \vec g + \rho^2 \tilde\Omega^2 \Jm^2 x).
\eeq 

We recall that the fluid's compressibility for isothermal conditions is
\beqs
\kappa=-\frac{1}{V}\frac{dV}{dp}=\frac{1}{\rho}\frac{d\rho}{dp},
\eeqs
where $V$, here, denotes the fluid's volume. In many cases such as (isothermal) compressible liquids, $\kappa$ is assumed to be a constant \cite{MuskatBook,BearBook}. In particular, it is a small positive constant for (isothermal) slightly compressible fluids such as crude oil and water. This condition is commonly used in petroleum and reservoir engineering \cite{AhmedHandbook2nd,Dakebook}, where the fluid dynamics in porous media have important applications. The current paper is focused on (isothermal)  slightly compressible fluids, hence, we study the following equation of state 
   \beq\label{slight}
\frac{1}{\rho}   \frac{d\rho}{dp}=\kappa, 
\quad \text{where the constant compressibility } \kappa>0 \text { is small}.
   \eeq

Using \eqref{slight}, we write \eqref{new2} as
\beq\label{rveq}
\rho v= -X(\kappa^{-1} \nabla \rho - \rho^2 \vec g+ \rho^2 \tilde\Omega^2 \Jm^2 x).
\eeq

Consider the equation of continuity
\beq\label{eq5}
\tilde \phi\frac{\partial \rho}{\partial t} +\nabla\cdot(\rho v)=0,
\eeq
where $t$ is the time variable.
Combining \eqref{eq5} with \eqref{rveq} gives
\beq\label{eq0}
\tilde \phi\frac{\partial \rho}{\partial t}=\nabla\cdot(X(\kappa^{-1} \nabla \rho - \rho^2 \vec g+ \rho^2 \tilde \Omega^2 \Jm^2 x)).
\eeq

In the rotating frame, the gravitational field $\vec g$ becomes $\vec g(t)=\tilde{\mathcal G} e_0(t)$,
with the gravitational constant $\tilde{\mathcal G}>0$ and smooth unit vector-valued function $e_0(t)$, for $t\in\R$.

We make a simple change of variable  $u=\rho/\kappa$. Then we obtain from \eqref{eq0} the partial differential equation (PDE)
\beq\label{ueq}
\phi u_t=  \nabla \cdot\Big (X\big(\nabla u+u^2 \mathcal Z(x,t)\big)\Big),  
\eeq
where
\beq\label{Zx}
\mathcal Z(x,t)=-\mathcal G e_0(t)+ \Omega^2  \Jm^2 x,
\eeq
\beqs
\phi =\kappa\tilde \phi,\quad
\mathcal G=\kappa^2\tilde{\mathcal G},\quad \Omega=\kappa \tilde \Omega.
\eeqs

To reduce the complexity in our mathematical treatment, hereafter, we consider the involved parameters and all equations to be non-dimensional. This is allowed by using appropriate scalings.

In this paper, we study the initial and boundary value problem (IBVP) for equation \eqref{ueq}.
More specifically, let $U$ be an open, bounded  set in $\R^3$ with $C^1$ boundary $\Gamma=\partial U$.
We study the following problem  
\beq\label{ibvpg}
\begin{aligned}
\begin{cases}
\phi u_t=\nabla\cdot \Big(X\big(\nabla u+u^2 \mathcal Z(x,t)\big)\Big)\quad &\text{in}\quad U\times (0,\infty)\\
u(x,0)=u_0(x) \quad &\text{in}\quad U\\
u(x,t)= \psi(x,t)\quad &\text{in}\quad \Gamma\times (0,\infty),
\end{cases}
\end{aligned}
\eeq
where the initial data $u_0(x)$ and the Dirichlet boundary data $\psi(x,t)$ are given.

We will focus on the mathematical analysis of problem \eqref{ibvpg}. 
We obtain various estimates of the solution in terms of the initial and boundary data.
These estimates show how the solutions, in space and time, can be controlled by the initial and boundary data. 
We emphasize that the dependence on the problem's key parameters, including the angular speed of rotation, are expressed explicitly in our results.  

The paper is organized as follows. 
In section \ref{Prem}, we establish basic properties of the function $X$ which are crucial to our understanding of problem \eqref{ibvpg}.
They reveal the nonlinear structure and the degeneracy of the nonlinear parabolic equation \eqref{ueq}. Moreover, they have explicit dependence on the physical parameters, which, as stated above, is an important goal of this paper.
In section \ref{maxsec}, we prove the maximum principle for non-negative solutions of equation \eqref{ueq} in Theorem \ref{maxpr}. Using this, we derive the maximum estimates for non-negative solutions of the IBVP \eqref{ibvpg} in Corollary \ref{maxcor}.
Section \ref{Gradprep} contains the Lady\v{z}enskaja--Ural{\cprime}ceva-typed embedding, Theorem \ref{LUembed}, with the weight $K[w,Q]$ which is related to the type of degeneracy of the nonlinear PDE \eqref{ueq}. This is one of the key tools in obtaining higher integrability for the gradient later.
In section \ref{L2asub}, we establish the estimate for the $L^{2-a}_{x,t}$-norm of the gradient in Theorem \ref{L2apos}. It was done through the $\KK$-weighted $L^2$-norm first, see Proposition \ref{L2a}, and then by the interpretation of the weight $\KK$.
In section \ref{gradsec}, we estimate the $L^{s}_{x,t}$-norms of the gradient, which is interior in the spatial variables,  for any finite number $s>2-a$. Specifically, we obtain estimates for $2-a<s\le 4-a$ in subsection \ref{L4ma}, and for $s>4-a$ in subsection \ref{Lhigher}.
We use the iteration method by Lady\v{z}enskaja and Ural{\cprime}ceva \cite{LadyParaBook68}. This is a classical technique but, with suitable modifications based on the structure of equation \eqref{ueq}, applies well to our complicated nonlinear PDE. Moreover, it is sufficiently explicit to allow us to track all the necessary constants.
Section \ref{maxintime} is devoted to the estimates for the the gradient's $L^\infty_tL^s_x$-norms, which, of course, are stronger norms than those in the previous section.

It is worth mentioning that the derived estimates in this paper are already complicated, therefore, we strive to make them coherent, and hence more digestible, rather than sharp.

Concerning the simplification \eqref{sim}, it is a common strategy when encountering a new nonlinear problem. As presented above, it allows us to formulate the whole fluid dynamics system as a scalar parabolic equation \eqref{ueq}. 
Such approximation, usually with some average density $\rho_*$, makes the problem much more accessible, while still gives insights on the flows' behaviors.
More importantly, this approach prompts the way to analyze the full model. Indeed, in the general case, the $\mathcal R$ in \eqref{Rdef} becomes $\mathcal R(u)$, and the PDE \eqref{ueq} becomes
\beq\label{unext}
\phi u_t=  \nabla \cdot\Big (X\big(u,\nabla u+u^2 \mathcal Z(x,t)\big)\Big),  
\eeq
with $X(u, y)$ defined in the same way as \eqref{Xdef}. 
Therefore, we can reduce the fluid dynamics system to a scalar PDE again. Furthermore, the properties of function $X$ established in subsection \ref{Xpropsec}  with explicit dependence on $\mathcal R=\mathcal R(u)$,  and other $X$-related results in section \ref{Gradprep} will play fundamental roles in understating the structure of the PDE \eqref{unext}.
This will be pursued and reported in a sequel of this paper.

Finally, we comment that the choice of the equation of state \eqref{slight}, in addition to its meaningful applications, yields the PDE \eqref{ueq} which can be analyzed rather thoroughly. Indeed, many strong estimates will be obtained in the next sections. In spite of this focus on the slightly compressible fluids, the techniques developed in the current paper can be combined with those in our previous work \cite{CHK1,CHK2} to model and analyze other types of compressible fluid flows such as the rotating isentropic flows for gases.

\section{Preliminaries}\label{Prem}

This section contains prerequisites and basic results on function $X$.

\subsection{Notation and elementary inequalities} 
A vector $x\in\R^3$ is denoted by a $3$-tuple $(x_1,x_2,x_3)$ and considered as a column vector, i.e., a $3\times 1$ matrix. Hence $x^{\rm T}$ is the $1\times 3$ matrix $(x_1\ x_2\ x_3)$.

For a function $f=(f_1,f_2,\ldots,f_m):\R^n\to\R^m$, its derivative is the $m\times n$ matrix
\begin{align*}
f'=Df=\Big(\frac{\partial f_i}{\partial x_j}\Big)_{1\le i\le m,\,1\le j\le n}.
\end{align*}

In particular, when $n=3$ and $m=1$, i.e.,  $f:\R^3\to \R$, the derivative is 
$$f'=Df=(f_{x_1}\ f_{x_2}\ f_{x_3}),$$ 
while its gradient vector is  
$$\nabla f=(f_{x_1},f_{x_2},f_{x_3})=(Df)^{\rm T}.$$
The Hessian matrix is $$D^2 f\eqdef D(\nabla f)=  \Big(\frac{\partial^2 f}{\partial x_j \partial x_i} \Big)_{i,j=1,2,3}.$$

Interpolation inequality for integrals:  
\beq\label{Lpinter}
\int |f|^sd\mu\le \Big(\int |f|^p d\mu\Big)^\frac{q-s}{q-p}\Big(\int |f|^q d\mu\Big)^\frac{s-p}{q-p}\text{ for }0<p<s<q.
\eeq

\medskip
For two vectors $x,y\in \R^3$, their dot product is $x\cdot y=x^{\rm T}y=y^{\rm T}x$, while $xy^{\rm T}$ is the $3\times 3$ matrix $(x_iy_j)_{i,j=1,2,3}$.

Let $\mathbf A=(a_{ij})$ and $\mathbf B=(b_{ij})$ be any $3\times 3$ matrices of real numbers. Their inner product is
\beqs 
\mathbf A:\mathbf B\eqdef {\rm trace}(\mathbf A\mathbf B^{\rm T})=\sum_{i,j=1}^3 a_{ij}b_{ij}.
\eeqs

The Euclidean norm of the matrix $\mathbf A$ is $$|\mathbf A|=(\mathbf A:\mathbf A)^{1/2}=\Big(\sum_{i,j=1}^3 a_{ij}^2\Big)^{1/2}.$$ 
(Note that we do not use $|\mathbf A|$ to denote the determinant in this paper.)

When $\mathbf A$ is considered as a linear operator, another norm is defined by
\beq\label{opnorm}
\|\mathbf A\|_{\rm op}=\max\left\{ \frac{|\mathbf Ax|}{|x|}:x\in\R^3,x\ne 0\right\} = \max\{ |\mathbf Ax|:x\in\R^3,|x|=1\}.
\eeq

We have the following inequalities
\begin{align}
\label{mm0}
|\mathbf A y|&\le |\mathbf A|\cdot |y| \quad\text{ for any  }y\in\R^3,\\
\label{mop}
|\mathbf A y|&\le \|\mathbf A\|_{\rm op}\cdot |y|\quad  \text{ for any  }y\in\R^3,\\
\label{mmi}
|\mathbf A \mathbf B|&\le |\mathbf A|\cdot |\mathbf B|.
\end{align}

It is also known that
\beq\label{nnorms}
\|\mathbf A\|_{\rm op}\le |\mathbf A|\le c_*\|\mathbf A\|_{\rm op},
\eeq
where $c_*$ is a positive constant independent of $\mathbf A$.

In particular, for matrix $\Jm$, we observe, for any $x\in\R^3$, that 
\beq\label{Jxineq}
|\Jm x|\le |\eR|\,|x|=|x|\quad\text{ and }\quad |\Jm^2x|\le |\Jm x|\le |x|.
\eeq
By choosing $x\ne 0$ perpendicular to $\eR$, we conclude, for the norm \eqref{opnorm}, that 
$$\|\Jm\|_{\rm op}=\|\Jm^2\|_{\rm op}=1.$$
For the Euclidean norm,  we have, from explicit formulas in \eqref{JR}, that
\beq\label{Jnorm}
|\Jm|^2=2|\eR|^2=2,
\eeq
\beq\label{J2norm}
|\Jm^2|^2=3-2|\eR|^2+|\eR|^4=2.
\eeq

\medskip
We recall below some more elementary inequalities that will be used in this paper. First,
\beq\label{ee3}
 \frac{x^p+y^p}2\le (x+y)^p\le 2^{(p-1)^+}(x^p+y^p)\quad  \text{for all } x,y\ge 0,\  p>0,
\eeq
where $z^+=\max\{z, 0\}$ for any $z\in\R$.
Particularly,
\beq\label{ee2}
(x+y)^p\le 2^p(x^p+y^p)\quad  \text{for all } x,y\ge 0,\ p>0.
\eeq
Second, 
\beq\label{ee4}
x^\beta \le x^\alpha+x^\gamma\quad \text{for all } x,y\ge 0,\ 0\le \alpha\le \beta\le\gamma,
\eeq
\beq\label{ee5}
x^\beta \le 1+x^\gamma \quad \text{for all } x,y\ge 0,\  0\le \beta\le\gamma.
\eeq
Above and throughout the paper, we conveniently use $0^0=1$.

By the triangle inequality and the second inequality of \eqref{ee3}, we have
\beq\label{ee6}
|x-y|^p\ge 2^{-(p-1)^+}|x|^p-|y|^p \quad \text{for all } x,y\in\R^n,\  p>0.
\eeq

\subsection{Properties of the function $X$}\label{Xpropsec}
It is obvious that the structure of the parabolic equation \eqref{ueq} depends greatly on the properties of the function $X$. 
Thus, this subsection is devoted to studying $X$. 

Recall that the functions $F$ and $X$ are defined in Definition \ref{FX}. 
Throughout the paper, we denote 
\beq \label{adef}
a=\frac{\alpha_N}{1+\alpha_N}\in(0,1),
\eeq 
\beq\label{defxione}
\chi_0=g(1)=\sum_{i=0}^N a_i\text{ and }\chi_1=g(1)+\mathcal R=\chi_0+\mathcal R.
\eeq

\begin{lemma}\label{lem21}

{\rm (i)} One has 
\beq\label{X0}
\frac{c_1\chi_1^{-1}|y|}{(1+|y|)^a}\le |X(y)|\le \frac{c_2\chi_1^a|y|}{(1+|y|)^a} \text{ for all }y\in \mathbb R^3,
\eeq
where 
$c_1=\min\{1,\chi_0\}^a $ and $c_2=2^a c_1^{-1}\min\{a_0,a_N\}^{-1}$.
Alternatively,
\beq\label{X1}
 \chi_1^{-(1-a)}|y|^{1-a}-1\le |X(y)|\le c_3|y|^{1-a} \text{ for all }y\in \mathbb R^3,
 \eeq
where $c_3=(a_N)^{a-1}$.

{\rm (ii)} One has
\beq\label{X2}
\frac{c_4 \chi_1^{-2} |y|^2}{(1+|y|)^a}  \le X(y)\cdot y\le \frac{c_2 \chi_1^a |y|^2}{(1+|y|)^a} \text{ for all }y\in \mathbb R^3,
\eeq
where $c_4=(\min\{1,a_0,a_N\}/2^{\alpha_N})^{1+a}$.
Alternatively,
\beq\label{X3}
c_5\chi_1^{-2} (|y|^{2-a}-1)  \le X(y)\cdot y\le c_3|y|^{2-a}\text{ for all }y\in \mathbb R^3,
\eeq
where $c_5=2^{-a}c_4$.
\end{lemma}
\begin{proof}
Let $y\in\R^3$ and $v=X(y)$. Then, by \eqref{Xdef}, 
\beq\label{Fvy} F(v)=y.\eeq 

(i) Since $g(|v|)v$ and $\eR\times v$ are orthogonal, we have from \eqref{Fdef} and \eqref{Fvy} that
\beqs
 |y|^2 = g(|v|)^2|v|^2+\mathcal R^2|\Jm v|^2.
\eeqs
This and \eqref{Jxineq} show that
\beqs
g(|v|)^2|v|^2\le |y|^2\le (g(|v|)^2+\mathcal R^2)|v|^2.
\eeqs
Thus,
\beq\label{gf1}
g(|v|)|v|\le |y|\le (g(|v|)+\mathcal R)|v|.
\eeq

\textit{Proof of \eqref{X1}.}
From the first inequality in \eqref{gf1},
\beq\label{vf}
|y| \ge g(|v|)|v|\ge a_N |v|^{\alpha_N+1}, \text{ which implies }|v|\le (a_N^{-1} |y|)^\frac{1}{\alpha_N+1}=c_3|y|^{1-a}.
\eeq
So we obtain the second inequality of \eqref{X1}.
From \eqref{gf1},
\beq\label{fov}
|y|\le \chi_1(1+|v|)^{\alpha_N+1}.
\eeq
Then we obtain the first inequality in \eqref{X1}.

\textit{Proof of \eqref{X0}.}
Since $X(0)=0$, we consider only $y,v\ne 0$.

\textit{Case $|v|>1$.} By \eqref{gf1}, $|y|> g(1)\cdot 1=\chi_0$. 
Furthermore, by \eqref{vf},
\beq\label{v1}
|v|\le (a_N^{-1}|y|)^{1-a}=\frac{2^a a_N^{a-1}|y|}{(|y|+|y|)^a}\le \frac{2^a \chi_1^a a_N^{-1}|y|}{(|y|+\chi_0)^a}\le \frac{2^a a_N^{-1}\chi_1^a|y|}{(\min\{1,\chi_0\})^a(|y|+1)^a}.
\eeq

On the other hand, we have from \eqref{gf1} that $|y|\le (g(1)|v|^{\alpha_N}+\mathcal R)|v|\le \chi_1 |v|^{\alpha_N+1}$. Then
\beq\label{v2}
|v|\ge (\chi_1^{-1}|y|)^{1-a}\ge \frac{\chi_1^{a-1}|y|}{|y|^a}\ge \frac{\chi_0^a \chi_1^{-1}|y|}{(|y|+1)^a}.
\eeq

\textit{Case $0<|v|\le 1$.} It follows \eqref{gf1} that
$a_0 |v|\le  |y|\le \chi_1 |v|\le \chi_1$.
Thus,
\begin{align}\label{v3}
|v|&\le a_0^{-1} |y|=\frac{a_0^{-1}|y|(|y|+\chi_0)^a}{(|y|+\chi_0)^a}\le \frac{a_0^{-1}|y|(\chi_1+\chi_0)^a}{(|y|+\chi_0)^a}
\le \frac{ a_0^{-1}(2\chi_1)^a|y|}{(\min\{1,\chi_0\})^a(|y|+1)^a},\\
\label{v4}
|v|&\ge \chi_1^{-1} |y|\ge \chi_1^{-1} |y| \frac{1}{(|y|+1)^a}.
\end{align}

From \eqref{v1} and \eqref{v3}, we obtain the second inequality in \eqref{X0}.
From \eqref{v2} and \eqref{v4}, we obtain the first inequality in \eqref{X0}.

(ii) Note that the second inequality of \eqref{X2}, respectively \eqref{X3}, follows the Cauchy-Schwarz inequality and the second inequality of \eqref{X0}, respectively \eqref{X1}. Thus, we focus on proving the first inequalities of \eqref{X2} and \eqref{X3}.
We have from \eqref{Fdef} and \eqref{gf1} that
\beqs
X(y)\cdot y=v\cdot F(v)=g(|v|)|v|^2
\ge  \frac{g(|v|)|y|^2}{(g(|v|)+\mathcal R)^2}  .
\eeqs

We estimate
\begin{align}
\label{g2}
g(|v|)+\mathcal R
&\le \chi_1(1+|v|)^{\alpha_N},\\
\intertext{and, by using \eqref{ee2},}
\label{g1}
g(|v|)
&\ge  \min\{a_0,a_N\} (1+|v|^{\alpha_N})
\ge \min\{a_0,a_N\} 2^{-\alpha_N}(1+|v|)^{\alpha_N}.
\end{align}
Hence, 
\begin{align*}
\frac{g(|v|)}{(g(|v|)+\mathcal R)^2}
\ge   \frac{\min\{a_0,a_N\} 2^{-\alpha_N}}{\chi_1^2(1+|v|)^{\alpha_N}}.
\end{align*}
Note, by \eqref{ee3}, that we also have 
\beqs
1+|y|\ge 1+g(|v|)|v|\ge \min\{1,a_N\}(1+|v|^{\alpha_N+1})
\ge \min\{1,a_N\}2^{-\alpha_N}(1+|v|)^{\alpha_N+1}.
\eeqs
Then,
\beq\label{vy}
(1+|v|)^{\alpha_N} \le  (1+|y|)^a (2^{\alpha_N}/\min\{1,a_N\})^a.
\eeq
Therefore,
\beq\label{gg}
\frac{g(|v|)}{(g(|v|)+\mathcal R)^2}
\ge \frac{\min\{a_0,a_N\} 2^{-\alpha_N}}{\chi_1^2(1+|y|)^a}(\min\{1,a_N\}2^{-\alpha_N})^a\ge \frac{c_4 \chi_1^{-2}}{(1+|y|)^a}.
\eeq
Hence we obtain the first inequality in \eqref{X2}.
Then the first inequality of \eqref{X3} follows this by considering $|y|\le 1$ and $|y|>1$ separately.
\end{proof}

\begin{remark} 
Compared to \eqref{X1}, inequality in \eqref{X0} indicates $X(y)\to0$ as $|y|\to0$ at the rate $|y|$, while $X(y)\to\infty$ as $|y|\to\infty$ at a different rate $|y|^{1-a}$. 
This refined form \eqref{X0} and its proof originate from \cite[Lemma 2.1]{CHIK1}. 
\end{remark}

\begin{lemma}\label{Fpinv}
The function $F$ belongs to $C^1(\R^3)$, and the derivative matrix $F'(v)$ is invertible for each $v\in\R^3$.
Consequently, $X\in C^1(\R^3)$ and $X'(y)$ is invertible for each $y\in\R^3$.
\end{lemma}
\begin{proof}
Elementary calculations show that
\begin{align}
 \label{Fprime}
F'(v)&=g'(|v|)\frac{v v^{\rm T}}{|v|}+g(|v|)\mathbf I_3+ \mathcal R \Jm\text{ for }v\ne 0,\\
\label{Fp0}
F'(0)&=g(0)\mathbf I_3+\mathcal R \Jm.
\end{align}
Clearly, $F'(v)$ is continuous at $v\ne 0$, and $F'(v)\to F'(0)$ as $v\to 0$. Therefore, $F\in C^1(\R^3)$.

For $z\in \R^3$, we have  
\begin{align*}
z^{\rm T} F'(v) z 
&=g'(|v|)\frac{(z^{\rm T}v)^2}{|v|}+g(|v|)|z|^2\text{ for }v\ne 0,\\
z^{\rm T} F'(0) z&=g(0)|z|^2.
\end{align*}
Since $g'(s)>0$ for $s>0$, it follows that
\beq\label{Fzz}
z^{\rm T} F'(v) z \ge g(|v|)|z|^2 \text{ for all } v,z\in\R^3.
\eeq

Let $v\in\R^3$. If $F'(v)z=0$, then $0=z^{\rm T} F'(v)z\ge g(|v|)|z|^2$, which implies that $z=0$. Hence, $F'(v)$ is invertible.

By the Inverse Function Theorem, the statements for $X$ follow those for $F$.
\end{proof}

\begin{lemma}\label{Xder}
For any $y\in\R^3$, the derivative matrix $X'(y)$ satisfies
\beq\label{Xprime}
c_6\chi_1^{-1}(1+|y|)^{-a}\le |X'(y)|\le  c_7(1+\chi_1)^a (1+|y|)^{-a},
\eeq
\beq\label{hXh} 
\xi^{\rm T} X'(y)\xi \ge c_8 \chi_1^{-2}(1+|y|)^{-a} |\xi|^2\text{ for all }\xi\in \R^3,
\eeq
where 
\begin{align*}
c_6= \sqrt3(2^{-\alpha_N}\min\{1,a_N\})^a/(\alpha_N+2),\quad c_7=c_* 2^{\alpha_N}/\min\{a_0,a_N\},\quad
c_8=c_4/(\alpha_N+2)^2. 
\end{align*}
\end{lemma}
\begin{proof}
Let $y\in\R^3$, then, by \eqref{Xdef},  $X'(y)=(F'(v))^{-1}$, with $F(v)=y.$
We first claim that
\beq\label{claim}
\frac{c_9}{g(|v|)+\mathcal R}\le |X'(y)|\le \frac{c_*}{g(|v|)},
\eeq
where $c_9=\sqrt3/(\alpha_N+2)$, and $c_*>0$ is the positive constant in \eqref{nnorms}.

Accepting \eqref{claim} for a moment, we prove the inequality \eqref{Xprime}.
Observe, by  \eqref{fov}, that
\beq\label{g5}
1+|y|\le (\chi_1+1)(1+|v|)^{\alpha_N}.
\eeq
On the one hand, \eqref{g1} and \eqref{g5} yield
\beq\label{g3}
g(|v|)
\ge \min\{a_0,a_N\} 2^{-\alpha_N} (1+|y|)^a/(1+\chi_1)^a.
\eeq
On the other hand, \eqref{g2} and \eqref{vy} give 
\beq\label{g4}
g(|v|)+\mathcal R\le \chi_1(1+|y|)^a (2^{\alpha_N}/\min\{1,a_N\})^a.
\eeq
Then, by combining \eqref{claim}, \eqref{g3} and \eqref{g4}, we obtain \eqref{Xprime}.

We now prove the claim \eqref{claim}.

\medskip
\noindent\textit{Proof of the first inequality in \eqref{claim}.} 
First, we consider $y,v\ne 0$. In \eqref{Fprime}, the matrix $g'(|v|)|v|^{-1}v v^{\rm T}+g(|v|)\mathbf I_3$ is symmetric, while $\mathcal R \Jm$ is anti-symmetric. Hence they are orthogonal, and, together with \eqref{Jnorm}, we have
\beqs
\begin{split}
|F'(v)|^2
&=\Big|g'(|v|)\frac{v v^{\rm T}}{|v|}+g(|v|)\mathbf I_3\Big|^2+\mathcal R^2 |\Jm|^2\\
&={\rm trace}\Big\{\big(g'(|v|)\frac{v v^{\rm T}}{|v|}+g(|v|)\mathbf I_3\big)^2\Big\}+2\mathcal R^2\\
&={\rm trace}\Big( (g'(|v|))^2v v^{\rm T} +2g'(|v|)g(|v|)\frac{v v^{\rm T}}{|v|} +(g(|v|))^2I_3\Big)+2\mathcal R^2.
\end{split}
\eeqs
Since trace$(vv^{\rm T})=|v|^2$, we have 
\begin{align*}
|F'(v)|^2
&=(g'(|v|))^2 |v|^2 +2g'(|v|)g(|v|)|v| +3(g(|v|))^2+2\mathcal R^2\\
&=(g'(|v|) |v| +g(|v|))^2 +2g(|v|)^2+2\mathcal R^2.  
\end{align*}

Note that
\beqs
g'(|v|)>0 \quad \text{ and }\quad 
g'(|v|)|v|\le \alpha_N g(|v|).
\eeqs

Then
\beq\label{i1}
|F'(v)|^2\le ((\alpha_N+1)^2+2)g(|v|)^2 +2\mathcal R^2
\le (\alpha_N+2)^2 (g(|v|)+\mathcal R)^2.
\eeq

Similarly, we have from \eqref{Fp0} that
\beq\label{i2}
|F'(0)|^2=3g(0)^2+ 2\mathcal R^2\le 3 (g(0)+\mathcal R)^2 \le (\alpha_N+2)^2(g(0)+\mathcal R)^2 .
\eeq

From \eqref{i1} and \eqref{i2}, it follows 
\beq\label{FpU}
|F'(v)|\le (\alpha_N+2)(g(|v|)+\mathcal R)\text{ for all }v\in \R^3.
\eeq

By \eqref{mmi}, $$\sqrt 3 = |\mathbf I_3| =| F'(v)X'(y) | \le |F'(v)|\cdot |X'(y)|,$$ which gives
\begin{align*}
|X'(y)|\ge \frac{\sqrt 3}{ |F'(v)|} \ge \frac{\sqrt3}{(\alpha_N+2)(g(|v|)+\mathcal R) }=\frac{c_9}{g(|v|)+\mathcal R}. 
\end{align*}

\medskip
\noindent\textit{Proof of  the second inequality in \eqref{claim}.} 
For $z\ne 0$, by Cauchy-Schwarz inequality and \eqref{Fzz}, we have
\beq\label{Fpz}
|F'(v)z|\ge \frac{|z^{\rm T} F'(v)z|}{|z|} \ge g(|v|)|z|.
\eeq
 
For any $\xi\in\mathbb R^3\setminus \{0\}$, applying \eqref{Fpz}  to $z=X'(y)\xi$, which is non-zero thanks to $X'(y)$ being invertible, gives
\begin{align*}
|\xi|\ge g(|v|)|X'(y)\xi|.
\end{align*}
Thus, the operator norm of $X'(y)$ is bounded by $\|X'(y)\|_{\rm op}\le 1/g(|v|)$,  and then, by relation \eqref{nnorms}, 
\beqs
|X'(y)|\le c_*\|X'(y)\|_{\rm op}\le c_*/g(|v|).
\eeqs

\medskip
\noindent\textit{Proof of \eqref{hXh}.}  
Let $u=X'(y)\xi=(F'(v))^{-1}\xi$, which gives $\xi=F'(v)u$.
Rewriting $[X'(y)\xi]\cdot \xi$ in terms of $u,v$ and using property \eqref{Fzz}, we have
\beq\label{Xh1}
[X'(y)\xi]\cdot \xi= u\cdot [ F'(v)u ]\ge |u|^2 g(|v|)=|X'(y)\xi|^2 g(|v|).
\eeq 
From \eqref{mm0} and \eqref{FpU}, for any $\xi\in\R^3$:
\beqs
|\xi|=|F'(v)X'(y)\xi|\le |F'(v)| \cdot |X'(y)\xi|\le (\alpha_N+2) (g(|v|)+\mathcal R)|X'(y)\xi|,
\eeqs
thus
\beq\label{Xh2}
|X'(y)\xi|\ge \frac{|\xi|}{(\alpha_N+2)(g(|v|)+\mathcal R)}.
\eeq
Combining \eqref{Xh1}, \eqref{Xh2} and \eqref{gg} yields
\begin{align*}
\xi^{\rm T}X'(y)\xi&\ge \frac{g(|v|)}{(\alpha_N+2)^2(g(|v|)+\mathcal R)^2} |\xi|^2\ge \frac{c_4\chi_1^{-2} |\xi|^2}{(\alpha_N+2)^2(1+|y|)^a}.
\end{align*}
This proves \eqref{hXh}.
\end{proof}

\section{Maximum estimates}\label{maxsec}

We will estimate the solutions of \eqref{ibvpg} by the maximum principle.
Denote \beq\label{Phi0} \hh(x,t)=\nabla u(x,t)+u^2(x,t)\mathcal Z(x,t).\eeq

We re-write the PDE in \eqref{ibvpg} in the non-divergence form as
\begin{align*}
\phi u_t=X'(\hh):(D\hh)^{\rm T}
&=X'(\hh):\big(D^2 u + u^2 D\mathcal Z(x,t)+ 2u \mathcal Z(x,t) Du\big)^{\rm T}\\
&=\mathbf A:\big (D^2 u + u^2 \Omega^2 \Jm^2 + 2u (\mathcal Z(x,t) Du)^{\rm T}\big ),
\end{align*}
where $\mathbf A=\mathbf A(x,t)= X'(\hh(x,t))$.

We write $\mathbf A=\mathbf S+\mathbf T,$ where $\mathbf S$ and $\mathbf T$ are the symmetric and anti-symmetric parts of $\mathbf A$, i.e.,
$$\mathbf S=(s_{ij})_{i,j=1,2,3}=\frac12(\mathbf A+\mathbf A^{\rm T})\quad \text{and}\quad 
\mathbf T=\frac12(\mathbf A-\mathbf A^{\rm T}).$$

Sine $D^2u$ is symmetric, we have 
$\mathbf T:D^2u=0$, hence $\mathbf A:D^2u=\mathbf S :D^2u$.

Similarly, $\Jm^2$ is symmetric, and we have $\mathbf A:\Jm^2=\mathbf S:\Jm^2$.

Therefore,
\beq\label{usym}
\phi u_t=\mathbf S:D^2 u+u^2 \Omega^2 \mathbf S:\Jm^2 +  2u \mathbf A:(\mathcal Z(x,t) Du)^{\rm T}.
\eeq

Equation \eqref{usym} turns out to possess a maximum principle.
Recall that the parabolic boundary of $U\times(0,T]$ is 
\beqs
\partial_p (U\times(0,T])=\bar U\times[0,T]\setminus U\times (0,T].
\eeqs

\begin{theorem}[Maximum principle]\label{maxpr}
Suppose $T>0$, and $u\in C(\bar U\times[0,T])\cap C_{x,t}^{2,1}(U\times (0,T])$ with $u\ge 0$ is a solution of \eqref{ueq} on $U\times (0,T]$. Then
 \beq\label{max-u}
 \max_{\bar U\times [0,T]} u =\max_{\partial_p (U\times(0,T])} u.
 \eeq
\end{theorem} 
\begin{proof}
We make use of equation \eqref{usym} which is equivalent to \eqref{ueq}. 
We examine the second term on the right-hand side of \eqref{usym}.
Direct calculations using the formula of $\Jm^2$  in \eqref{JR} give
$$\mathbf S:\Jm^2
=\sum_{i,j=1}^3 k_ik_j s_{ij} -\sum_{i=1}^3 s_{ii}=\eR^{\rm T}\mathbf S\eR -\rm{trace}(\mathbf S).$$

By \eqref{hXh}, we have
$\xi^{\rm T}\mathbf A\xi \ge 0$ for all $\xi\in\R^3$,
and, hence,
\beq\label{Axx}
\xi^{\rm T}\mathbf S\xi \ge 0\text{ for all }\xi\in\R^3.
\eeq

By  \eqref{Axx} and the fact $\mathbf S$ is symmetric, we have $\mathbf S\ge 0$ with eigenvalues $\lambda_1,\lambda_2,\lambda_3\ge 0$. Then, applying Cauchy-Schwarz's inequality and \eqref{mop} to $\eR^{\rm T}\mathbf S \eR$, we obtain 
\beq \label{newpos}
\mathbf S:\Jm^2 \le \|\mathbf S\|_{\rm op} |\eR|^2 -\rm{trace}(\mathbf S)=\max\{\lambda_1,\lambda_2,\lambda_3\}\cdot 1 - (\lambda_1+\lambda_2+\lambda_3)\le 0.
\eeq 

Let $\varep>0$. Set  $u^\varep(x,t)=e^{-\varep t}u(x,t)$ and 
$\displaystyle M_\varep=\max_{\bar U\times [0,T]} u^\varep$.
We prove that
 \beq\label{maxe}
\max_{\bar U\times [0,T]} u^\varep = \max_{\partial_p (U\times (0,T])} u^\varep.
 \eeq

Suppose \eqref{maxe} is false. Then $M_\varep>0$ and there exists a point $(x_0,t_0)\in U\times (0,T]$ for such that 
$u^\varep(x_0,t_0)=M_\varep$. At this maximum point $(x_0,t_0)$ we have 
\beq\label{vt1}
u^\varep_t (x_0,t_0)\ge 0,
\quad \nabla u^\varep(x_0,t_0)=0,\quad   D^2 u^\varep(x_0,t_0)\le 0.
\eeq

We observe the followings:

(a) The second property of \eqref{vt1} implies $\nabla u(x_0,t_0)=0$.

(b) On the one hand, we have from \eqref{Axx} that $\mathbf S(x,t)\ge 0$ on $U\times(0,T]$.
On the other hand, the last property of \eqref{vt1} implies $D^2 u(x_0,t_0)\le 0$.
Then it is well-known that $\mathbf S(x_0,t_0): D^2u(x_0,t_0)\le 0$, see e.g. \cite[Chapter 2, Lemma 1]{FriedmanPara2008}.

From \eqref{usym}, \eqref{newpos}, and (a), (b), we obtain $u_t(x_0,t_0)\le 0$.
Therefore,
\beqs
u^\varep_t(x_0,t_0)=-\varep e^{-\varep t_0} u(x_0,t_0) +e^{-\varep t_0} u_t(x_0,t_0)
\le -\varep  u^\varep(x_0,t_0)=-\varep M_\varep <0.
\eeqs
This contradicts the first inequality in \eqref{vt1}, hence,  \eqref{maxe} holds true.
 Note that
 \beqs
e^{-\varep T}\max_{\bar U\times [0,T]} u\le \max_{\bar U\times [0,T]} u^\varep = \max_{\partial_p (U\times (0,T])} u^\varep\le \max_{\partial_p (U\times (0,T])} u
\le \max_{\bar U\times [0,T]} u.
 \eeqs
Then letting $\varep\to 0$ yields \eqref{max-u}.
\end{proof}

\begin{corollary}\label{maxcor}
 Let $u\in C(\bar U\times[0,T])\cap C_{x,t}^{2,1}(U\times (0,T])$ with $u\ge 0$ solve problem \eqref{ibvpg} on a time interval $[0,T]$ for some $T>0$. Then it holds for all $t\in[0,T]$ that
\beq\label{uM} 
\max_{x\in\bar U}u(x,t)\le M_0(t)\eqdef \sup_{x\in U} |u_0(x)|+\sup_{(x,\tau)\in \partial U \times (0,t]} |\psi(x,\tau)|.
\eeq
 \end{corollary} 
 \begin{proof}
  Because of the continuity of $u$ on $\bar U\times [0,t]$, the quantity $M_0(t)$, in fact, is an upper bound of the maximum of $u$ on $\partial_p(U\times(0,t])$. Then applying inequality \eqref{max-u} to $T=t$ yields estimate \eqref{uM}. 
 \end{proof}

\section{Preparations for the gradient estimates}\label{Gradprep}

This section contains technical preparations for the estimates for different norms of the gradient $\nabla u$ in the next three  sections.

Given two mappings $Q:U\to\R^3$ and  $w:U\to \R$, we define a function $K[w,Q]$ on $U$ by
\beq\label{Ku}
K[w,Q](x)=(1+|\nabla w(x) + w^2(x) Q(x) |)^{-a}\quad\text{ for } x\in U.
\eeq

This function will be conveniently used in comparisons with $X(\nabla w(x) + w^2(x) Q(x))$ arising in the PDE \eqref{ueq}.

\begin{lemma}\label{kuglem}
If $s\ge a$, then the following inequalities hold on $U$:
\begin{align}
\label{kug1}
 K[w,Q]|\nabla w|^s 
 &\le 2^{2s-a}|\nabla w|^{s-a}+2^{2s+1-a}(1+|w^2 Q|^s),\\
\label{kug2}
K[w,Q]|\nabla w|^s&\ge 3^{-1}|\nabla w|^{s-a} - 3^{-1} (1+|w^2 Q|^s).
\end{align}
\end{lemma}
\begin{proof}
We denote, in this proof, $K=K[w,Q]$. 
Let $s\ge a$, by the triangle inequality and inequalities \eqref{ee2}, \eqref{ee5}, we have
\begin{align*}
 K|\nabla w|^s
 &\le 2^s K\cdot (|\nabla w +w^2 Q|^s+|w^2 Q|^s)\\
 &\le 2^s |\nabla w +w^2 Q|^{s-a}+ 2^s K|w^2 Q|^s\\
 &\le 2^s\cdot 2^{s-a}( |\nabla w|^{s-a} + |w^2 Q|^{s-a})+2^s K|w^2 Q|^s.
\end{align*}
Using inequality \eqref{ee5} and the fact $K\le 1$, we estimate
\beqs
|w^2 Q|^{s-a}\le 1+|w^2 Q|^s, \quad K|w^2 Q|^s\le 1+|w^2 Q|^s. 
\eeqs 
Hence,
\beq\label{Kw1}
 K|\nabla w|^s
\le  2^{2s-a}|\nabla w|^{s-a} +(2^{2s-a}+2^s)(1+|w^2 Q|^s).
\eeq
Noticing that $2^{2s-a}+2^s\le 2\cdot 2^{2s-a}$, we obtain \eqref{kug1} from \eqref{Kw1}.

To prove \eqref{kug2}, we write $ |\nabla w|^{s-a} = K|\nabla w|^{s-a} (1+|\nabla w+w^2 Q| )^a$, and apply 
inequality \eqref{ee3} to have
\begin{align*}
 |\nabla w|^{s-a} &\le K|\nabla w|^{s-a} (1+|\nabla w|^a+|w^2 Q|^a )
 =K\cdot (|\nabla w|^{s-a} +|\nabla w|^s+|\nabla w|^{s-a}|w^2 Q|^a ).
\end{align*}

Concerning the last sum between the parentheses, applying inequality \eqref{ee5} to its first term gives 
\begin{align*}
|\nabla w|^{s-a}\le 1 + |\nabla w|^s,
\end{align*}
and applying Young's inequality to its last term with the conjugate powers $s/(s-a)$ and $s/a$, when $s>a$, gives
\begin{align*}
|\nabla w|^{s-a}|w^2 Q|^a\le |\nabla w|^s+|w^2 Q|^s.
\end{align*}
Obviously, this inequality also holds when $s=a$. Thus,
\beq\label{pre-kug}
|\nabla w|^{s-a} 
\le 3K|\nabla w|^s+K\cdot(1+|w^2 Q|^s)
\le 3K|\nabla w|^s+(1+|w^2 Q|^s).
\eeq
We obtain \eqref{kug2}.
\end{proof}

For our later convenience, we rewrite inequality \eqref{pre-kug}, by replacing $s-a$ with $s$,  as
\beq\label{kugs}
|\nabla w|^s\le 3K[w,Q]|\nabla w|^{s+a} + (1+|w^2 Q|^{s+a}) \text{ for all }s\ge 0.
\eeq

\begin{theorem}\label{LUembed}
For each $s\geq1$, there exists a constant $C>0$ depending only on $s$ and number $a$ in \eqref{adef} such that  for any function $w\in C^2(U)$ and non-negative function 
$\zeta\in C_c^1(U)$, the following inequality holds 
\beq\label{LUineq}
\begin{split}
\int_U K[w,Q] |\nabla w|^{2s+2} \zeta^2 dx
&\le
 C \sup_{\text{\rm supp} \zeta }(w^2) \Big \{
 \int_U K[w,Q] |D^2 w|^2 (|\nabla w|^{2s-2} +1)\zeta^2 dx\\
&\quad + \int_U K[w,Q] |\nabla w|^{2s}  \Big(|\nabla \zeta|^2 +(|w||Q'|+|Q|)^2\zeta^2 w^2 \Big) dx\Big\} \\
&\quad +C\int_U \Big( 1+  | w^2 Q|^{2s}\Big)  | w^2Q |^{2s+2}  \zeta^2 dx.
\end{split}
\eeq

Assume, in addition, that $Q$ and $Q'$ are bounded on $U$. Then,
\beq\label{ibt}
\begin{split}
&\int_U K[w,Q] |\nabla w|^{2s+2} \zeta^2 dx
\le
 C M_w^2 \int_U K[w,Q] |D^2 w|^2 |\nabla w|^{2s-2}  \zeta^2 dx\\
&\quad + CM_w^2 \int_U K[w,Q]|\nabla w|^{2s}  (|\nabla \zeta|^2 + (M_w\mu_Q+M_Q)^2 M_w^2 \zeta^2 ) dx \\
&\quad +  C{\rm sgn}(s-1) M_w^2 \int_U K[w,Q] |D^2 w|^2 \zeta^2 dx
+  C |U|(M_w^2 M_Q)^{2s+2}(1+M_w^2 M_Q)^{2s},
\end{split}
\eeq
where 
\begin{align*}
M_w=\sup_{x\in\text{\rm supp} \zeta } |w(x)|,\quad M_Q=\sup_{x\in U} |Q(x)|,\quad \mu_Q=\sup_{x\in U} |Q'(x)|.
\end{align*}
\end{theorem}
\begin{proof}
For convenience in computing the derivatives, we will first establish \eqref{LUineq} with $K[w,Q]$ being replaced by the following function
\beq\label{Kstar}
K_*(x)= (1+ |\nabla w(x) + w^2(x)Q(x)|^2 )^{-a/2}.
\eeq

In this proof, the symbol $C$ denotes a generic positive constant depending only on $s$ and number $a$ in \eqref{adef}, while $C_\varep>0$ depends on $s$, $a$ and $\varep$.

We use Einstein's summation convention in our calculations. 
Let 
$$I =  \int_U K_*(x) |\nabla w(x)|^{2s+2} \zeta^2(x) dx.$$
By integration by parts and direct calculations, we see that 
\begin{align*}
I 
&=\int_U (K_*  |\nabla w|^{2s} \partial_i w \zeta^2)\cdot \partial_i w dx
 = - \int_U \partial_i (K_*  |\nabla w|^{2s} \partial_i w \zeta^2)\cdot  w dx\\
 &=I_1+I_2+I_3+I_4,
\end{align*}
where
\begin{align*}
I_1& = a \int_U \frac{(\nabla w + w^2Q) \cdot (\partial_i \nabla w +  2w\partial_i w Q+ w^2  \partial_i Q)   }{(1+ |\nabla w + w^2Q|^2)^{a/2+1}} |\nabla w|^{2s} \partial_i w\cdot  \zeta^2 \cdot w dx,\\
I_2&= - \int_U K_* |\nabla w|^{2s} \Delta w \cdot \zeta^2  \cdot w dx,\\
I_3&= - 2s \int_U K_*\cdot \Big(|\nabla w|^{2s-2} \partial_i\partial_j w \partial_j w\Big)\cdot  \partial_i w \cdot \zeta^2  \cdot w dx,\\
I_4&= -2 \int_U K_* |\nabla w|^{2s} \partial_i w \zeta \partial_i  \zeta \cdot w dx.
\end{align*}

Above, we used $\partial_i(|\nabla w|^{2s})=\partial_i(|\nabla w|^2)^s=s(|\nabla w|^2)^{s-1}\partial_i (|\nabla w|^2)$ to avoid possible singularities when $\nabla w=0$.

We estimate $I_1$ first. Observe that
\beqs
\frac{1}{(1+ |\nabla w + w^2Q|^2)^{a/2+1}} 
= \frac{K_*}{1+|\nabla w + w^2Q|^2}\le\frac{2K_*}{(1+|\nabla w + w^2Q|)^2} .
\eeqs

By Cauchy-Schwarz and triangle inequalities, we have
\begin{align*}
I_1
&\le C \int_U \frac{K_* }{(1+|\nabla w + w^2Q|)^2} \cdot |\nabla w + w^2Q|\\
&\qquad \cdot \Big (|D^2 w| |\nabla w|^{2s+1}  +  |w||\nabla w|^{2s+2}|Q|+ w^2 |\nabla w|^{2s+1}|Q'| \Big)   \zeta^2 \cdot |w| dx \\
&\le  J_1+J_2+J_3,
\end{align*}
where
\begin{align*}
 J_1&= C\int_U \frac{K_* }{1+|\nabla w + w^2Q|}(|D^2 w| |\nabla w|^{2s+1}  )   \zeta^2 \cdot |w| dx,\\
 J_2&= C\int_U \frac{K_* }{1+|\nabla w + w^2Q|}( |wQ||\nabla w|^{2s+2})   \zeta^2 \cdot |w| dx,\\
 J_3&= C\int_U \frac{K_* }{1+|\nabla w + w^2Q|}(w^2 |Q'||\nabla w|^{2s+1} )   \zeta^2 \cdot |w| dx.
\end{align*}

For $J_1$, by using triangle inequality
$|\nabla w|\le  |\nabla w+ w^2  Q |+ |w^2  Q|,$
we have
\beqs
\begin{split}
J_1 &\le  C\int_U \frac{K_* }{1+|\nabla w + w^2  Q|}(|\nabla w+ w^2  Q |^{2s+1} + |w^2  Q|^{2s+1})  |D^2 w|    \zeta^2 |w| dx\\
&\le C  \int_U K_* \cdot \Big( |\nabla w+ w^2  Q |^{2s}+    |w^2  Q|^{2s+1}\Big)  |D^2 w|    \zeta^2 |w| dx.
\end{split}
\eeqs
Applying triangle inequality and \eqref{ee2} gives
$$ |\nabla w+ w^2  Q |^{2s}\le C( |\nabla w|^{2s}+ |w^2  Q|^{2s}).$$
Thus,
\beqs
J_1 \le C \int_U K_*  |\nabla w|^{2s}  |D^2 w|    \zeta^2 |w| dx + C \int_U K_* \cdot \Big( |w^2  Q |^{2s}+  |w^2  Q|^{2s+1}\Big)  |D^2 w|    \zeta^2 |w| dx.
\eeqs

Let $\varep>0$. Denote
\beqs
I_5=\int_U K_* |D^2 w|^2|\nabla w|^{2s-2}\zeta^2 w^2 dx,\quad
I_6=\int_U K_* |D^2 w|^2\zeta^2 w^2 dx.
\eeqs

In the last inequality for $J_1$, we apply Cauchy's inequality to obtain
\beqs
\begin{split}
J_1 &\le C \int_U K_*  |\nabla w|^{2s}  |D^2 w|    \zeta^2 |w| dx + C \int_U K_* \cdot \Big( |w^2  Q |^{2s}+  |w^2  Q|^{2s+1}\Big)  |D^2 w|    \zeta^2 |w| dx\\
&\le \varep I+ C_\varep I_5 +  C \int_U K_* \cdot \Big( |w^2  Q |^{4s}+  |w^2  Q|^{4s+2}\Big)    \zeta^2 dx
+CI_6.
\end{split}
\eeqs

Similarly, 
\beqs
\begin{split}
J_2 &\le C \int_U \frac{K_* }{1+|\nabla w + w^2  Q|}(|\nabla w+ w^2  Q |^{2s+2} + |w^2  Q|^{2s+2})  \zeta^2 w^2 |Q|dx\\
&\le C  \int_U K_* \cdot \Big( |\nabla w+ w^2  Q |^{2s+1}+    |w^2  Q|^{2s+2}\Big)    \zeta^2 w^2|Q| dx\\
&\le  C \int_U K_* \cdot \Big( |\nabla w|^{2s+1}+ |w^2  Q |^{2s+1}+  |w^2  Q|^{2s+2}\Big)   \zeta^2 w^2|Q| dx.
\end{split}
\eeqs

For $J_3$, neglecting the denominator in the integrand gives
\beqs
J_3 \le  C\int_U K_* |\nabla w|^{2s+1}  \zeta^2 |w|^3|Q'| dx.
\eeqs

Combining the above estimates for $J_1$, $J_2$ and $J_3$ yields 
\begin{align*}
 I_1\le \varep I+ C_\varep I_5 +CI_6+ J_4+J_5,
\end{align*}
where
\begin{align*}
 J_4&= C\int_U K_* |\nabla w|^{2s+1}  (|Q|+|w||Q'|) \zeta^2 w^2 dx,\\
J_5&=  C \int_U K_* \cdot \Big( |w^2  Q |^{4s}+  |w^2  Q|^{4s+2}\Big)    \zeta^2 dx
+ C\int_U K_* \cdot \Big( |w^2  Q |^{2s+1}+  |w^2  Q|^{2s+2}\Big)   \zeta^2 w^2 |Q|dx.
\end{align*}

We estimate, by using Cauchy's inequality,
\beqs
J_4
\le \varep I+C_\varep\int_U K_* |\nabla w|^{2s}   (|Q|+|w||Q'|)^2 \zeta^2 w^4 dx,
\eeqs
and, with $K_*\le 1$,
\begin{align*}
 J_5&\le  C \int_U \Big( |w^2  Q |^{4s}+  |w^2  Q|^{4s+2}\Big)    \zeta^2 dx
+ C\int_U \Big( |w^2  Q |^{2s+1}+  |w^2  Q|^{2s+2}\Big)   \zeta^2 w^2 |Q|dx\\
&=C\int_U ( |w^2 Q |^{4s}+  |w^2 Q|^{4s+2}+ |w^2 Q |^{2s+2}+  |w^2 Q|^{2s+3}) \zeta^2 dx.
\end{align*}
Applying inequality \eqref{ee4} gives
\beqs 
J_5\le    C\int_U \Big( | w^2Q |^{2s+2}+  | w^2 Q|^{4s+2}\Big)    \zeta^2 dx
= C\int_U \Big( 1+  | w^2 Q|^{2s}\Big)  | w^2Q |^{2s+2}  \zeta^2 dx.
\eeqs

Therefore,
\begin{align*}
 I_1&\le 2\varep I+ C_\varep I_5 +CI_6+C_\varep\int_U K_* |\nabla w|^{2s} (|Q|+|w||Q'|)^2  \zeta^2  w^4 dx\\
&\quad + C\int_U \Big( 1+  | w^2 Q|^{2s}\Big)  | w^2Q |^{2s+2}  \zeta^2 dx.
\end{align*}

The terms $I_2$, $I_3$, $I_4$ can be bounded simply  by
\begin{align*}
I_2+I_3+I_4
&\le \int_U K_* |\nabla w|^{2s} |\Delta w| \zeta^2  |w| dx
 + 2 s \int_U K_* |\nabla w|^{2s} |D^2  w|  \zeta^2  |w| dx\\
&\quad +2 \int_U K_* |\nabla w|^{2s+1}  \zeta |\nabla \zeta| |w|dx\\
&\le C\int_U K_*  |\nabla w|^{2s} |D^2 w | \zeta^2 |w| dx
+2\int_U K_*  |\nabla w|^{2s+1}  \zeta|\nabla \zeta|\, |w| dx.
\end{align*}
Applying Cauchy's inequality to each integral gives
\begin{align*}
I_2+I_3+I_4
& \le (\varep I + C_\varep I_5)+\Big(\varep I +C_\varep \int_U K_*  |\nabla w|^{2s}  |\nabla \zeta|^2 w^2 dx\Big).
\end{align*}

Combining the estimates of $I_1$ and $I_2+I_3+I_4$, we have
\beqs
I \le  4\varep I +C_\varep I_5+CI_6
+  C_\varep I_7 + C\int_U \Big( 1+  | w^2 Q|^{2s}\Big)  | w^2Q |^{2s+2}  \zeta^2 dx,
\eeqs
where 
\begin{align*}
I_7=\int_U K_* |\nabla w|^{2s}  \Big(|\nabla \zeta|^2 +(|w||Q'|+|Q|)^2\zeta^2 w^2 \Big)w^2 dx.
\end{align*}

Selecting $\varep=1/8$, we obtain
\beq\label{Ks}
I \le  C(I_5+I_6+I_7)
 +C\int_U \Big( 1+  | w^2 Q|^{2s}\Big)  | w^2Q |^{2s+2}  \zeta^2 dx.
\eeq

In each integral of $I_5$, $I_6$ and $I_7$, we bound
\begin{align*}
\zeta^2 w^2\le \zeta^2\sup_{\rm supp \zeta}(w^2),\quad |\nabla \zeta|^2 w^2\le |\nabla \zeta|^2 \sup_{\rm supp \zeta}(w^2).
\end{align*}
It then follows \eqref{Ks} that
\beq\label{LU0}
\begin{split}
\int_U K_* |\nabla w|^{2s+2} \zeta^2 dx
&\le
 C \sup_{\text{\rm supp} \zeta }(w^2) \Big \{
 \int_U K_* |D^2 w|^2 (|\nabla w|^{2s-2} +1)\zeta^2 dx\\
&\quad + \int_U K_* |\nabla w|^{2s}  \Big(|\nabla \zeta|^2 +(|w||Q'|+|Q|)^2\zeta^2 w^2 \Big) dx\Big\} \\
&\quad +C\int_U \Big( 1+  | w^2 Q|^{2s}\Big)  | w^2Q |^{2s+2}  \zeta^2 dx.
\end{split}
\eeq

We compare $K_*$ in \eqref{Kstar} with $K[w,Q]$ in \eqref{Ku}.
Because 
\beqs
2^{-1}(1+ |\nabla w + w^2 Q|)^2 \le 1+ |\nabla w+ w^2 Q|^2 \le (1+ |\nabla w + w^2 Q|)^2,
\eeqs
then
\beq\label{KK}
K[w,Q](x)\le K_*(x)\le 2^{a/2} K[w,Q](x)\quad\forall x\in U.
\eeq

Applying the first, respectively second, inequality of \eqref{KK} to the left-hand side, respectively right-hand side, of \eqref{LU0}, we obtain inequality \eqref{LUineq}.

Now, consider the case  $w$, $Q$ and $Q'$ are bounded. By simple estimates of the last two integrals of \eqref{LUineq} using the numbers $M_w$, $M_Q$, $\mu_Q$, and by using the following estimate 
$$|\nabla w|^{2s-2}+1\le 2|\nabla w|^{2s-2}+{\rm sgn}(s-1)$$
for the first integral on the right-hand side of \eqref{LUineq},
we obtain inequality  \eqref{ibt}.
\end{proof}

\section{Gradient estimates (I) }\label{L2asub}

This section is focused on \emph{a priori} estimates for the gradient of a solution $u(x,t)$ of the IBVP \eqref{ibvpg}.

Hereafter, we fix $T>0$. Let $u(x,t)$ be a $C_{x,t}^{2,1}(\bar U\times(0,T])\times C(\bar U\times[0,T])$ solution of \eqref{ibvpg}, not necessarily non-negative.

In the estimates of the Lebesgue norms below, we will use the energy method. For that, it is convenient to shift the solution to a function vanishing at the boundary.

Let $\Psi(x,t)$ be the extension of $\psi(x,t)$ from $\Gamma\times(0,T]$ to $\bar U\times[0,T]$. It is assumed to have necessary regularity needed for calculations below. All our following estimates, as far as the boundary data is concerned, will be expressed in terms of $\Psi$. This will not lose the generality since we can always translate them into $\psi$-dependence estimates, see e.g. \cite{HI1,JerisonKenig1995}.

Define $\bar u=u-\Psi$ and $\bar u_0=u_0-\Psi$.
We derive from \eqref{ibvpg} the equations for $\bar u$:
\beq\label{ubar}
\begin{aligned}
\begin{cases}
\phi \bar u_t=\nabla\cdot (X(\Phi(x,t))-\phi \Psi_t\quad &\text{on}\quad U\times (0,T],\\
\bar u(x,0)=\bar u_0(x) \quad &\text{on}\quad U,\\
\bar u(x,t)= 0\quad &\text{on}\quad \Gamma\times (0,T],
\end{cases}
\end{aligned}
\eeq
where $\hh(x,t)$ is the same notation as \eqref{Phi0}, i.e.,
$$\hh(x,t)=\nabla u(x,t)+u^2(x,t)\mathcal Z(x,t)=\nabla u(x,t)+\ZZ(x,t),$$
with $\ZZ(x,t)=u^2(x,t)\mathcal Z(x,t).$

The following ``weight'' function will play important roles in our statements and proofs:
$$\KK(x,t)=K[u(\cdot,t),\mathcal Z(\cdot,t)](x)= (1+|\hh(x,t)|)^{-a}.$$

We estimate the $L^{2-a}_{x,t}$-norm for $\nabla u$ in terms of the initial and boundary data.
Define 
\beq\label{MZ}
\MZ=\sup_{\bar U\times[0,T]}|\mathcal Z(x,t)|,
\eeq
\beq\label{Estar}
\Ecal_0=\intTU (\chi_1^{2(2+a)}|\nabla \Psi|^2+ \phi\chi_1^2 (|\Psi_t|^2+\Psi^2))dxdt.
\eeq

\begin{xnotation}
In the remaining of this paper, the constant $C$ is positive and generic with varying values in different places.
It depends on number $N$, the coefficients $a_i$'s and powers $\alpha_i$'s of the function $g$ in \eqref{eq2}, the number $c_*$ in \eqref{nnorms}, and the set $U$.
In sections \ref{gradsec} and \ref{maxintime}, it further depends on number $s$, the subsets $U'$, $V$, $U_k$'s of $U$.
However, it does not depend on the initial and boundary data of $u$, the functions $\Phi(x,t)$, $\mathcal Z(x,t)$, and numbers $T$, $T_0$, $t_0$, $\phi$, $\mathcal R$, $\chi_*$, $M_*$, whenever these are introduced. In particular, it is independent of the cut-off function $\zeta$ in Lemmas \ref{lem61}, \ref{lem62}, \ref{lem67}, Proposition \ref{prop63}, and inequality \eqref{iterate2}.
\end{xnotation}

\begin{proposition}\label{L2a} One has
 \begin{align}
 \intTU \KK|\nabla u|^2 dxdt
 &\le  C\Big\{ \chi_1^2 \phi (\|\bar u_0\|_{L^2}^2+\|u\|_{L^2(U\times (0,T))}^2)\notag\\
 &\quad\qquad+(1+\chi_1)^{2(2+a)}\MZ^2\|u\|_{L^4(U\times (0,T))}^4+\Ecal_0\Big\},\label{gradu0}\\
 \intTU |\nabla u|^{2-a} dxdt
 &\le  C \Big\{\chi_1^2 \phi (\|\bar u_0\|_{L^2}^2+\|u\|_{L^2(U\times (0,T))}^2) \notag\\
 &\quad\qquad +(1+\chi_1)^{2(2+a)}(T+\MZ^2\|u\|_{L^4(U\times (0,T))}^4)+\Ecal_0\Big\}. \label{gradu1}
 \end{align}
\end{proposition}
\begin{proof}
In this proof, we denote $J= \int_U \Psi_t \bar u dx$.
Multiplying the PDE \eqref{ubar} by $\bar u$, integrating over domain $U$ and using  integration by parts, we have 
 \begin{align*}
\frac{\phi }{2} \frac{d}{dt} \int_U \bar u^2 dx
&=- \int_U X(\hh)\cdot  \nabla \bar u dx-\phi J\\
&=-\int_UX(\hh)\cdot \hh dx
+\int_UX(\hh)\cdot (\nabla \Psi + \ZZ) dx-\phi  J.
\end{align*}

By \eqref{X0} and \eqref{X2}, we have 
\begin{align*}
\frac\phi 2  \frac{d}{dt} \int_U \bar u^2   dx&\le -c_4\chi_1^{-2}\int_U \KK |\hh|^2 dx
+c_2\chi_1^a \int_U \KK |\hh|(|\nabla \Psi|+|\ZZ|)dx-\phi  J.
\end{align*}
For the second integral on the right-hand side, applying Cauchy's inequality gives
\begin{align*}
\frac\phi 2  \frac{d}{dt} \int_U \bar u^2   dx&\le -c_4\chi_1^{-2}\int_U \KK |\hh|^2 dx\\
&\quad+\int_U \Big[\varep\KK|\hh|^2 +C\varep^{-1}\chi_1^{2a}\KK(|\nabla \Psi|+|\ZZ|)^2\Big]dx-\phi  J.
\end{align*}

Choosing $\varep=c_4\chi_1^{-2}/2$, we obtain 
\begin{align}\label{ewJ}
\frac\phi 2  \frac{d}{dt} \int_U \bar u^2   dx&\le -\frac{c_4\chi_1^{-2}}{2}\int_U \KK |\hh|^2 dx+C \chi_1^{2(1+a)}\int_U  (|\nabla \Psi|^2+|\ZZ|^2)dx-\phi  J.
\end{align}

Note that 
\beq\label{Wbound}
|\ZZ|\le \MZ u^2.
\eeq

Writing $\bar u$ in $J$ as $u-\Psi$, and using Cauchy's inequality, we have
\beq\label{pJ}
\phi |J|\le \int_U \phi |\Psi_t| (|u|+|\Psi|) dx\le \phi \int_U (u^2+|\Psi_t|^2+|\Psi|^2)dx.
\eeq

Utilizing estimates \eqref{Wbound} and \eqref{pJ} in \eqref{ewJ}, we derive
\begin{align*}
\frac\phi 2  \frac{d}{dt} \int_U \bar u^2   dx
+\frac{c_4\chi_1^{-2}}{2}\int_U \KK|\hh|^2 dx
&\le C\chi_1^{2(1+a)}\MZ^2\int_U u^4 dx\\
&\quad+C \chi_1^{2(1+a)}\int_U  |\nabla \Psi|^2dx
+\phi\int_U (u^2 + |\Psi_t|^2+\Psi^2)dx.
\end{align*}

Then integrating from $0$ to $T$ gives
\beq
\begin{aligned}\label{KP}
 \intTU \KK|\hh|^2 dx dt
 &\le  C \chi_1^2 \phi \|\bar u_0\|_{L^2}^2 +C\chi_1^{2(2+a)}\MZ^2\intTU u^4 dxdt+C\phi \chi_1^2\intTU u^2 dxdt\\
&\quad  +C\Ecal_0.
 \end{aligned}
\eeq

For the left-hand side of \eqref{KP}, we use inequalities \eqref{ee6} and \eqref{Wbound} to have 
\begin{align*}
\KK|\hh|^{2}=\KK|\nabla u+\ZZ|^2\ge 2^{-1}\KK |\nabla u|^2-\KK|\ZZ|^2\ge 2^{-1} \KK|\nabla u|^2-\MZ^2 u^4.
\end{align*}

Combining this with \eqref{KP} yields 
\begin{align*}
 \intTU \KK|\nabla u|^2 dx
 &\le  2\intTU \KK|\hh|^2 dx dt + 2\MZ^2\intTU u^4 dxdt\\
 &\le  C \chi_1^2 \phi \|\bar u_0\|_{L^2}^2 +C(1+\chi_1)^{2(2+a)}\MZ^2\intTU u^4 dxdt\\
 &\quad+C\phi\chi_1^2\intTU u^2 dxdt
  +C\Ecal_0,
 \end{align*}
which proves \eqref{gradu0}.
 
Using \eqref{kug2} and \eqref{Wbound}, we estimate the integrand on the left-hand side of \eqref{gradu0} by
  \begin{align*}
  \KK|\nabla u|^2\ge 3^{-1}|\nabla u|^{2-a} -(1+|\ZZ|^2)\ge 3^{-1}|\nabla u|^{2-a} -(1+(u^2\MZ)^2). 
  \end{align*}
  It results in
\begin{align*}
 \intTU |\nabla u|^{2-a} dxdt
 &\le  3\intTU \KK|\nabla u|^2 dx + 3\intTU (1+\MZ^2 u^4) dxdt\\
 &\le  C \chi_1^2 \phi (\|\bar u_0\|_{L^2}^2+\|u\|_{L^2(U\times (0,T))}^2) \\
 &\quad+C(1+\chi_1)^{2(2+a)}\intTU (1+\MZ^2 u^4) dxdt+ C\Ecal_0,
 \end{align*}
which proves  \eqref{gradu1}.
 \end{proof}

To have more specific estimates, we examine the bounds for the constituents of the PDE in \eqref{ubar}.
Note, from \eqref{Zx} and \eqref{Jxineq},  that  
\beqs
|\mathcal Z(x,t)|
\le \mathcal G + \Omega^2 |\Jm^2 x|
\le \mathcal G + \Omega^2 r_0
=\mathcal G(1+\Omega_*^2),
\eeqs
where 
\beq \label{r0}
r_0=\max\{|x|: x\in \bar U \}
\text{ and }
\Omega_*=\Omega \sqrt{r_0/\mathcal G}=\tilde\Omega\sqrt{ r_0/\tilde{\mathcal G}}.
\eeq 
Thus, the number $\MZ$ in \eqref{MZ} can be bounded by
\beq\label{MZ1}
 \MZ\le \mathcal G(1+\Omega_*)^2.
\eeq

Next, it is obvious that 
$D\mathcal Z(x,t)=\Omega^2 \Jm^2$,
hence, by \eqref{J2norm},
\beq\label{DZ}
|D\mathcal Z(x,t)|=\muz\eqdef \sqrt 2 \Omega^2=\sqrt 2(\mathcal G/r_0) \Omega_*^2.
\eeq

We rewrite $\mathcal R$ in \eqref{Rdef} as
\beq\label{RR1}
\mathcal R=\frac{2\rho_* \Omega}{\phi}
=\frac{2\rho_* \sqrt{\mathcal G/r_0}}{\phi} \Omega_*
=\frac{2\rho_* \sqrt{ \tilde{\mathcal G}/r_0}}{\tilde\phi} \Omega_*.
\eeq

From \eqref{MZ1}, \eqref{DZ} and \eqref{RR1}, we conveniently relate the upper bounds of $\mathcal R$, $\MZ$, $\muz$ to a single parameter $\chi_*$ as follows
\beq\label{RMm}
\mathcal R\le \chi_*\eqdef \max\{1,d_*(1+\Omega_*)\},\quad \MZ\le \chi_*^2,\quad \muz\le \chi_*^2,
\eeq
where
\beqs
d_*
=\sqrt{\mathcal G/r_0}\max\Big\{\frac{2\rho_*}{\phi},\sqrt{r_0},2^{1/4}\Big\}.
\eeqs

The reason for \eqref{RMm}, with the choice of $\chi_*\ge 1$, is to simplify many estimates that will be obtained later, and specify the dependence of those estimates on the parameters $d_*$ and $\Omega_*$.

\begin{remark}
A very common situation is that the rotation is about the vertical axis, then $\eR=\pm (0,0,1)$, and
\beq\label{Jmatrix}
\Jm=\pm \begin{pmatrix}   0 & -1 & 0\\ 1 & 0 & 0 \\ 0& 0 & 0 \end{pmatrix},\quad 
\Jm^2=-\begin{pmatrix}
     1&0 & 0\\
     0&1&0\\
     0& 0&0
    \end{pmatrix}.
\eeq
Thanks to \eqref{Jmatrix}, we can, in this case, replace the $r_0$ in \eqref{r0} with a smaller number, namely, 
$$r_0=\max\{(x_1^2+x_2^2)^{1/2}: x=(x_1,x_2,x_3)\in \bar U \}.$$
\end{remark}

\begin{definition}\label{BEN}
 We will use the following quantities in our estimates throughout the paper:
\begin{align*}
  M_*&=\sup_{ U\times[0,T]} |u|, \quad
\Ecal_*=\intTU (|\nabla \Psi|^2+\phi |\Psi_t|^2+\phi \Psi^2)dxdt,\\
\Ncal_0&=\phi \|\bar u_0\|_{L^2}^2 + TM_*^2+\Ecal_*,\quad
\Ncal_*=\phi\|\bar u_0\|_{L^2}^2+T+ \Ecal_*,\\
\Ncal_2&=\phi( \|\bar u_0\|_{L^2}^2+\|\nabla u_0\|_{L^2}^2)+T+ \Ecal_*,\\
\Ncal_s&=\phi( \|\bar u_0\|_{L^2}^2+\|\nabla u_0\|_{L^2}^2+\|\nabla u_0\|_{L^s}^s)+T+ \Ecal_* \text{ for $s>2$.} 
\end{align*}
\end{definition}

The estimates obtained in the rest of this paper will depend on the quantities in Definition \ref{BEN}. Among those, only $M_*$ still depends on the solution $u$. 
However, this quantity can be bounded in terms of initial and boundary data by using different techniques.
For instance, in our original problem, $u=\rho/\kappa\ge 0$, hence we have from Corollary \ref{maxcor} that
$M_*\le C M_0(T)$, see \eqref{uM}.
Therefore, in the following, we say ``the estimates are expressed in terms of the initial and boundary data'' even when they contain $M_*$.

\medskip
As stated in section \ref{intro}, we will keep tracks of the dependence on certain physical parameters, particularly, the angular speed. Note, by \eqref{defxione} and \eqref{RMm}, that 
\beq\label{xichi}
\chi_1\le C\chi_*.
\eeq
With this and the fact $\chi_*\ge 1$, we compare $\Ecal_0$ in \eqref{Estar} with $\Ecal_*$ by
\beq\label{EEs}
\Ecal_0\le C\chi_*^{2(2+a)}\Ecal_*.
\eeq
It is clear that
\begin{align}
\label{NNN0}
\Ecal_*&\le \Ncal_0\le (M_*+1)^2 \Ncal_*,\\ 
\label{NNNs}
\Ecal_*&\le \Ncal_*\le \Ncal_2\le \Ncal_s \text{ for $s>2$.} 
\end{align}

\begin{theorem}\label{L2apos} 
One has
\begin{align}
\label{gradu2}
 \intTU \KK|\nabla u|^2 dxdt&\le C \Big\{\chi_*^2\phi \|\bar u_0\|_{L^2}^2 
 + \chi_*^{2(4+a)}TM_*^2(M_*+1)^2+\chi_*^{2(2+a)}\Ecal_*\Big\},\\
 \label{gradu3}
 \intTU |\nabla u(x,t)|^{2-a}dxdt &\le C\Big\{ \chi_*^2\phi \|\bar u_0\|_{L^2}^2 
 + \chi_*^{2(4+a)}T(M_*+1)^4+\chi_*^{2(2+a)}\Ecal_*\Big\}.
\end{align}
Consequently, the following more concise estimates hold
\beq 
\label{gradu4}
 \intTU \KK|\nabla u|^2 dxdt\le C\chi_*^{2(4+a)}(M_*+1)^2 \Ncal_0,
\eeq
\beq 
\label{gradu6a}
 \intTU \KK|\nabla u|^2 dxdt \le C\chi_*^{2(4+a)}(M_*+1)^4 \Ncal_*,
\eeq
\beq
\label{gradu6b}
 \intTU |\nabla u(x,t)|^{2-a}dxdt \le C\chi_*^{2(4+a)}(M_*+1)^4 \Ncal_*.
\eeq
\end{theorem}
\begin{proof}
 By the definition of $M_*$, we obviously have
\beq
|u|\le M_*\text{ on } \bar U\times[0,T],
\eeq
hence,
\beq\label{uu}
\|u\|_{L^2(U\times (0,T))}^2\le CTM_*^2\quad\text{and}\quad
\|u\|_{L^4(U\times (0,T))}^4\le CTM_*^4.
\eeq
Using \eqref{uu}, the relations \eqref{xichi}, \eqref{EEs}, and estimate of $\MZ$ in \eqref{RMm} for the right-hand side of \eqref{gradu0}, we have
\begin{align*}
 \intTU \KK|\nabla u|^2 dxdt
 &\le C\chi_*^2 \phi (\|\bar u_0\|_{L^2}^2 +TM_*^2) +C(1+\chi_*)^{2(2+a)}\cdot \chi_*^4\cdot  T M_*^4 +C\chi_*^{2(2+a)}\Ecal_*\\
 &\le C\chi_*^2 \phi \|\bar u_0\|_{L^2}^2 +C\chi_*^{2(2+a)+4} T M_*^2  (1+M_*^2) +C\chi_*^{2(2+a)}\Ecal_*.
\end{align*}
Then \eqref{gradu2} and \eqref{gradu4} follow. From \eqref{gradu4}, we infer \eqref{gradu6a} thanks to the last relation in \eqref{NNN0}.
 
Similarly, we have from  \eqref{gradu1} that 
\beqs
\intTU |\nabla u(x,t)|^{2-a}dxdt \le C\Big\{ \chi_*^2\phi (\|\bar u_0\|_{L^2}^2+TM_*^2)+\chi_*^{2(2+a)}(T+ \chi_*^4 \cdot T M_*^4) +\chi_*^{2(2+a)}\Ecal_*\Big\}.
\eeqs 
Then \eqref{gradu3} and \eqref{gradu6b} follow.
\end{proof}

We remark that while \eqref{gradu3} gives a direct estimate for the $L^{2-a}$-norm, the alternative estimate in form of \eqref{gradu2}  prepares for the iterations in section \ref{gradsec} below.

\begin{remark}\label{smlrmk1}
From the point of view of pure PDE estimates, the right-hand side of \eqref{gradu4} can be small, for a fixed $T>0$, while the right-hand side of \eqref{gradu6a}
and  \eqref{gradu6b} cannot. It is because $|u|$, $|\bar u_0|$, and $\Ecal_*$ being small will make $\Ncal_0$, but not $\Ncal_*$, small.
\end{remark}

\section{Gradient estimates (II)} \label{gradsec}

In this section, we establish the interior
$L^s$-estimate of $\nabla u $ for all $s > 0$.

For the remainder of this paper, $\zeta$ always denotes a function in $C^2_c(U\times [0,T])$ that satisfies
$0\le \zeta(x,t)\le 1$ for all $(x,t)$. 
When such a function is specified, the quantity $\mathcal D_s$ is defined, for $s\ge 0$, by
\beq\label{Ds}
\mathcal D_s= \int_U|\nabla u_0(x)|^{2s+2}\zeta^2(x,0) dx.
\eeq

The next two lemmas \ref{lem61} and \ref{lem62} are the main technical steps for the later iterative estimates of the gradient.

\begin{lemma} \label{lem61}
For any $s\ge 0$, one has
\beq\label{iterate1}
\begin{aligned}
&\phi \sup_{t\in[0,T]}\int_U |\nabla u(x,t)|^{2s+2} \zeta^2(x,t) dx,\\
&(s+1)c_8\chi_1^{-2}\intTU \KK  |D^2 u|^2   |\nabla u|^{2s}  \zeta^2  dx dt\le \phi \mathcal D_s +CI_0, 
\end{aligned}
\eeq
where 
\begin{align*}
I_0&=\muz^2 (1+\chi_1)^{2(1+a)}\intTU \KK  |\nabla u|^{2s} u^4   \zeta^2 dx dt \\
&\quad + (1+\chi_1)^{2(1+a)}\intTU \KK   |\nabla u|^{2s+2} (\MZ^2 u^2\zeta^2 + |\nabla \zeta|^2)   dx dt + \intTU |\nabla u|^{2s+2} \zeta|\zeta_t| dx dt.
\end{align*}
Consequently,
\begin{align}
 \label{irat0}
\phi \sup_{t\in[0,T]}\int_U |\nabla u(x,t)|^{2s+2} \zeta^2(x,t) dx
&\le \phi \mathcal D_s +C J_0,\\
\label{iterate4}
\intTU \KK  |D^2 u|^2   |\nabla u|^{2s}  \zeta^2  dx dt
&\le C \chi_*^2 (\phi \mathcal D_s+J_0),
\end{align}
where 
\begin{align*}
J_0
&=\chi_*^{2(3+a)}\widetilde M^4\intTU \KK  |\nabla u|^{2s}  \zeta^2 dx dt \\
&\quad +  \chi_*^{2(3+a)}(\widetilde M+1)^2\intTU \KK   |\nabla u|^{2s+2} (\zeta^2 + |\nabla \zeta|^2)   dx dt
 + \intTU |\nabla u|^{2s+2} \zeta|\zeta_t| dx dt,
\end{align*}
with 
\beq\label{Mtil}
\widetilde M=\sup\{ |u(x,t)|:(x,t)\in \supp\zeta\}.
\eeq
\end{lemma}
\begin{proof} 
For each $j=1,2,3$ and $t\in(0,T)$, let $\{\varphi_n(x)\}_{n=1}^\infty$ be a sequence in $C_c^\infty(U)$ that approximates $|\nabla u(x,t)|^{2s} \partial_j u(x,t)\zeta^2(x,t)$ in $W_0^{1,2}(U)$. Multiplying equation \eqref{ueq} by $-\partial_j \varphi_n(x)$, integrating the resulting equation over $U$, and using the integration by parts twice for the right-hand side give
\begin{align*}
-\phi\int_U \frac{\partial u}{\partial t} \partial_j\varphi_n dx
&=  -\int_U \frac{\partial}{\partial x_i}\Big[ X_i(\hh(x,t)) \Big] \partial_j \varphi_n dx
=\int_U X_i(\hh(x,t))  \partial_i \partial_j \varphi_n dx\\
&=  -\int_U \frac{\partial}{\partial x_j}\Big[ X_i(\hh(x,t)) \Big] \partial_i \varphi_n dx.
\end{align*}

Passing $n\to\infty$ and summing in $j$ yield
\begin{align*}
-\phi\int_U \frac{\partial u}{\partial t}\nabla\cdot ( |\nabla u|^{2s} \nabla u\zeta ^2)dx
&=  -\int_U \frac{\partial}{\partial x_j}\Big[ X_i(\hh(x,t)) \Big] \partial_i (|\nabla u|^{2s} \partial_j u \zeta^2) dx.
\end{align*}

Performing integration by parts again for the left-hand side, we obtain
\beq\label{esti}
\begin{aligned}
& \frac{\phi}{2s+2} \ddt \int_U |\nabla u|^{2s+2} \zeta^2 dx\\
&= -\int_U \frac{\partial}{\partial x_j}\Big[ X_i(\hh(x,t)) \Big] \partial_i (|\nabla u|^{2s} \partial_j u \zeta^2) dx+\frac{1}{s+1}\int_U |\nabla u|^{2s+2} \zeta\zeta_t dx\\
&=I_1+I_2+I_3+I_4,
\end{aligned}	
\eeq
where
\begin{align*}
 I_1&= -\int_U \frac{\partial}{\partial x_j}\Big[ X_i(\hh(x,t)) \Big]  \partial_j \partial_i u  |\nabla u|^{2s}  \zeta^2 dx,\\
I_2&= -2\int_U \frac{\partial}{\partial x_j}\Big[ X_i(\hh(x,t)) \Big]  \partial_j u \, |\nabla u|^{2s} \, \zeta \partial_i  \zeta dx,\\
I_3&= -2s\int_U \frac{\partial}{\partial x_j}\Big[ X_i(\hh(x,t)) \Big] \partial_j u   \, (|\nabla u|^{2s-2} \partial_i\partial_l u\partial_l u)\, \zeta^2 dx,\\
I_4&=\frac{1}{s+1} \int_U|\nabla u|^{2s+2} \zeta\zeta_t dx.
\end{align*}

For $i,j=1,2,3$, denote $X'_{ij}(y)=\partial X_i(y)/\partial y_j$. By the second inequality of \eqref{Xprime},
\beq\label{XpPh}
|X'_{ij}(\hh(x,t))|\le C(1+\chi_1)^a\KK(x,t).
\eeq

Let $I_5=\int_U \KK |D^2 u|^2   |\nabla u|^{2s}  \zeta^2 dx$.

\medskip
\noindent\textit{Estimation of $I_1$.} We have
\begin{align*}
I_1
&= -\int_U X'_{im}(\hh) ( \partial_m\partial _j u + \partial _j \mathcal Z_m u^2 + 2u \partial_j u \mathcal Z_m )    \partial_j \partial_i u  |\nabla u|^{2s}  \zeta^2 dx\\
&=I_{1,1}+I_{1,2}+I_{1,3},
\end{align*}	
where
\begin{align*}
I_{1,1}&=-\int_U D(\partial_j u)  X'(\hh) \nabla (\partial_j u )  |\nabla u|^{2s}  \zeta^2 dx,\\
I_{1,2}&=-\int_U X'_{im}(\hh)  \partial _j \mathcal Z_m u^2   \partial_j \partial_i u  |\nabla u|^{2s}  \zeta^2 dx,\\
I_{1,3}&=-2 \int_U X'_{im}(\hh) u \partial_j u \mathcal Z_m     \partial_j \partial_i u  |\nabla u|^{2s}  \zeta^2 dx.
\end{align*}

To estimate $I_{1,1}$, by applying \eqref{hXh} to $y=\hh$, $\xi=\nabla(\partial_j u)$, we have
\begin{align*}
 I_{1,1}
&\le - c_8\chi_1^{-2} \int_U \KK  \Big( \sum_{j=1}^n|\nabla(\partial_j u)|^2 \Big)  |\nabla u|^{2s}  \zeta^2 dx
=-c_8\chi_1^{-2} I_5. 
 \end{align*}
 
To estimate $I_{1,2}$, using  inequality \eqref{XpPh} to estimate $|X'_{im}(\hh)|$, identity \eqref{DZ} for $|\partial_j\mathcal Z_m|$, and then applying the Cauchy inequality to $u^2  |D^2 u |$, we obtain
\begin{align*}
|I_{1,2}|&\le C \int_U  (1+\chi_1)^a \KK  \muz u^2  |D^2 u | |\nabla u|^{2s}  \zeta^2 dx\\
&\le \varep I_5+ C\varep^{-1}(1+\chi_1)^{2a} \muz^2 \int_U \KK  u^4   |\nabla u|^{2s} \zeta^2 dx  .
\end{align*}

Similarly, we estimate $I_{1,3}$ by 
\begin{align*}
|I_{1,3}|&\le C(1+\chi_1)^a \int_U  \KK  \MZ u  |D^2 u | |\nabla u|^{2s+1}  \zeta^2 dx\\
&\le \varep I_5+ C\varep^{-1}(1+\chi_1)^{2a} \MZ^2\int_U \KK  u^2   |\nabla u|^{2s+2} \zeta^2 dx  .
\end{align*}

Summing up, we obtain 
\beq\label{estione}
\begin{aligned}
I_1&\le (2\varep -c_8\chi_1^{-2} )I_5+ C\varep^{-1}(1+\chi_1)^{2a} \muz^2 \int_U \KK  u^4   |\nabla u|^{2s} \zeta^2 dx\\
&\quad + C\varep^{-1}(1+\chi_1)^{2a} \MZ^2\int_U \KK  u^2   |\nabla u|^{2s+2} \zeta^2 dx.  
\end{aligned}	
\eeq

\medskip
\noindent\textit{Estimation of $I_2$.} We calculate
\begin{align*}
I_2
&=-2\int_U X'_{im}(\hh)\Big ( \partial_m\partial _j u + \partial _j \mathcal Z_m u^2 + 2u \partial_j u \mathcal Z_m\Big )   \partial_j u \, |\nabla u|^{2s} \, \zeta \partial_i  \zeta dx\\
&=I_{2,1}+I_{2,2}+I_{2,3},
\end{align*}
where the integral is split along the sum $\partial_m\partial _j u + \partial _j \mathcal Z_m u^2 + 2u \partial_j u \mathcal Z_m$.

Using \eqref{XpPh} and Cauchy's inequality gives 
\begin{align*}
|I_{2,1}| &\le C(1+\chi_1)^a\int_U  \KK  |D^2 u||\nabla u|^{2s+1} \zeta |\nabla  \zeta| dx\\
&\le \varep I_5+C\varep^{-1}(1+\chi_1)^{2a} \int_U \KK  |\nabla u|^{2s+2}  |\nabla  \zeta|^2 dx,
\end{align*}

Using \eqref{XpPh}, \eqref{DZ} and \eqref{MZ}, we obtain
\begin{align*}
|I_{2,2}| &\le C(1+\chi_1)^a\int_U  \KK |D\mathcal Z| u^2|\nabla u|^{2s+1} \zeta |\nabla  \zeta| dx\\
&\le C(1+\chi_1)^a\muz \int_U  \KK u^2|\nabla u|^{2s+1} \zeta |\nabla  \zeta| dx,
\end{align*}
and
\begin{align*}
|I_{2,3}| &\le C(1+\chi_1)^a\int_U \KK |\mathcal Z| u|\nabla u|^{2s+2} \zeta |\nabla  \zeta| dx\\
&\le C(1+\chi_1)^a\MZ\int_U \KK  u|\nabla u|^{2s+2} \zeta |\nabla  \zeta| dx.
\end{align*}

Thus, we have
\beq\label{estitwo}
\begin{aligned}
&I_2\le \varep I_5+C\varep^{-1}(1+\chi_1)^{2a} \int_U \KK  |\nabla u|^{2s+2}  |\nabla  \zeta|^2 dx\\
&\quad+C(1+\chi_1)^a\muz \int_U  \KK u^2|\nabla u|^{2s+1} \zeta |\nabla  \zeta| dx+C(1+\chi_1)^a\MZ\int_U \KK  u|\nabla u|^{2s+2} \zeta |\nabla  \zeta| dx.
\end{aligned}
\eeq

\medskip
\noindent\textit{Estimation of $I_3$.} Using similar calculations to those for $I_2$, we have 
\begin{align*}
I_3
&=-2s\int_U X'_{mi}(\hh) \Big( \partial_m\partial _j u + \partial _j \mathcal Z_m u^2 + 2u \partial_j u \mathcal Z_m \Big)
\cdot \partial_j u   \, (|\nabla u|^{2s-2} \partial_i\partial_l u\partial_l u)\,     \zeta^2 dx\\
&=I_{3,1}+I_{3,2}+I_{3,3}.
\end{align*}
where the integral, again,  is split along the sum $\partial_m\partial _j u + \partial _j \mathcal Z_m u^2 + 2u \partial_j u \mathcal Z_m$.

Rewriting $I_{3,1}$ and applying \eqref{hXh} to $y=\hh$ and $\xi=\frac{1}2 \nabla( |\nabla u|^2)$, we have
\begin{align*}
I_{3,1}
&=-2s\int_U X'_{im}(\hh)\cdot  [ \partial_m\partial _j u \cdot \partial_j u ] \cdot[\partial_i\partial_l \cdot u\partial_l u] \cdot |\nabla u|^{2s-2}\,     \zeta^2 dx\\
&=-2s\int_U X'_{im}(\hh)  \Big [\frac{1}2 \partial _m  |\nabla u|^2\Big ]  \Big[\frac{1}2 \partial_i |\nabla u|^2)\Big]\cdot  |\nabla u|^{2s-2}     \zeta^2 dx\\
&\le -2s\, c_8 \chi_1^{-2} \int_U  \KK  \Big|\frac{1}2\nabla  (|\nabla u|^2)\Big|^2   |\nabla u|^{2s-2}      \zeta^2 dx \le 0. 
\end{align*}

We estimate $I_{3,2}$ and $I_{3,3}$ similarly to $I_{2,1}$, $I_{2,2}$, $I_{2,3}$, and obtain
\begin{align*}
|I_{3,2}|&\le C(1+\chi_1)^a\int_U \KK  |D\mathcal Z| u^2    \, |\nabla u|^{2s} \, |D^2 u|    \zeta^2 dx\\
&\le \varep I_5 + C\varep^{-1}(1+\chi_1)^{2a} \muz^2 \int_U \KK  u^4    \, |\nabla u|^{2s}    \zeta^2 dx,
\end{align*}
and 
\begin{align*}
|I_{3,3}|&\le C(1+\chi_1)^a \MZ\int_U \KK  u    \, |\nabla u|^{2s+1} \, |D^2 u|   \zeta^2 dx\\
&\le \varep I_5 + C\varep^{-1} (1+\chi_1)^{2a} \MZ^2\int_U \KK  u^2    \, |\nabla u|^{2s+2}    \zeta^2 dx.
\end{align*}

Therefore,
\beq\label{estithree}
\begin{aligned}
I_3&\le  2\varep I_5+ C\varep^{-1}(1+\chi_1)^{2a} \muz^2 \int_U \KK  u^4    \, |\nabla u|^{2s}    \zeta^2 dx\\
&\quad + C\varep^{-1} (1+\chi_1)^{2a} \MZ^2\int_U \KK  u^2    \, |\nabla u|^{2s+2}    \zeta^2 dx.
\end{aligned}
\eeq

Combining \eqref{esti} with the estimates \eqref{estione}, \eqref{estitwo} and \eqref{estithree}, we have 
\begin{align*}
 &\frac{\phi}{2s+2} \ddt \int_U |\nabla u|^{2s+2} \zeta^2 dx +(c_8\chi_1^{-2}-5\varep)I_5 
 \le  C\varep^{-1}(1+\chi_1)^{2a}\muz^2 \int_U \KK  u^4   |\nabla u|^{2s} \zeta^2 dx  \\
 & +C\varep^{-1} (1+\chi_1)^{2a}\MZ^2\int_U \KK  u^2   |\nabla u|^{2s+2} \zeta^2 dx  
 +C\varep^{-1}(1+\chi_1)^{2a} \int_U \KK |\nabla u|^{2s+2}  |\nabla  \zeta|^2 dx\\
 & +C(1+\chi_1)^a\muz\int_U \KK |u|^2|\nabla u|^{2s+1} \zeta |\nabla  \zeta| dx
 + C(1+\chi_1)^a\MZ \int_U \KK |u||\nabla u|^{2s+2} \zeta |\nabla  \zeta| dx 
 +I_4.
\end{align*}

Using the Cauchy inequality, we have, for the fourth integral on the right-hand side,
\beqs
\begin{split}
(1+\chi_1)^a\muz |u|^2|\nabla u|^{2s+1} \zeta |\nabla  \zeta|
&\le \frac 1 2 \left [\varep^{-1}(1+\chi_1)^{2a}\muz^2  u^4   |\nabla u|^{2s} \zeta^2+ \varep |\nabla u|^{2s+2}  |\nabla  \zeta|^2  \right],
\end{split}
\eeqs
and, for the fifth integral on the right-hand side,  
\beqs
\begin{split}
(1+\chi_1)^a\MZ|u| \zeta |\nabla  \zeta|
&\le \frac 1 2\left[ \varep^{-1} (1+\chi_1)^{2a}\MZ^2 u^2\zeta^2+ \varep  |\nabla  \zeta|^2 \right].
\end{split}
\eeqs

Therefore, we obtain
\beq \label{d10}
\begin{aligned}
& \frac{\phi}{2s+2} \ddt \int_U |\nabla u|^{2s+2} \zeta^2 dx +(c_8\chi_1^{-2}-5\varep)I_5 \\
& \le  C\varep^{-1}(1+\chi_1)^{2a}\muz^2 \int_U \KK  |\nabla u|^{2s} u^4   \zeta^2 dx  
   +C\varep^{-1} (1+\chi_1)^{2a}\MZ^2\int_U \KK  u^2   |\nabla u|^{2s+2} \zeta^2 dx  \\
&\quad  +C(\varep^{-1}(1+\chi_1)^{2a}+\varep) \int_U \KK |\nabla u|^{2s+2}  |\nabla  \zeta|^2 dx
  + C \int_U |\nabla u|^{2s+2} \zeta|\zeta_t| dx. 
\end{aligned} 
\eeq 

Choosing $\varep=c_8\chi_1^{-2}/10$, and integrating \eqref{d10} in time, we get 
$$\phi \sup_{t\in[0,T]}\int_U |\nabla u(x,t)|^{2s+2} \zeta^2(x,t) dx \text{ and } (s+1)c_8\chi_1^{-2}\intTU \KK  |D^2 u|^2   |\nabla u|^{2s}  \zeta^2  dx dt$$
are bounded from above by  
\begin{align*}
&\phi \int_U |\nabla u_0(x)|^{2s+2} \zeta^2(x,0) dx 
+C\muz^2 (1+\chi_1)^{2(1+a)}\intTU \KK  |\nabla u|^{2s} u^4   \zeta^2 dx dt \\
&\quad + C ((1+\chi_1)^{2(1+a)}+\varep)\intTU \KK   |\nabla u|^{2s+2} (\MZ^2 u^2\zeta^2 + |\nabla \zeta|^2)   dx dt \\
&\quad + C\intTU |\nabla u|^{2s+2} \zeta|\zeta_t| dx dt.
\end{align*}
Then using the fact $\varep\le c_8 \chi_0^{-2}/10\le C$ for the last $\varep$, we obtain \eqref{iterate1}.

We now estimate $I_0$ further. We use \eqref{RMm}, \eqref{xichi} to estimate $\MZ$, $\muz$, $\chi_1$, and note, by \eqref{Mtil}, that 
$u^m \zeta^2\le \widetilde M^m \zeta^2$ for $m=2,4$.
With these estimates, we have 
\beq \label{JJ0}I_0\le C J_0.
\eeq 
Hence, \eqref{irat0} directly follows the first estimate \eqref{iterate1}.

Similarly,  multiplying the second estimate in \eqref{iterate1} by $(s+1)^{-1}c_8^{-1}\chi_1^2$, then using  \eqref{xichi} and \eqref{JJ0}, we obtain \eqref{iterate4}.
\end{proof}

Next, we combine Lemma \ref{lem61} with the embedding in Theorem \ref{LUembed} to derive a bootstrapping estimate.

\begin{lemma}\label{lem62}
If $s\ge 0$ then 
\beq \label{Kug3}
\intTU \KK  |\nabla u|^{2s+4} \zeta^2 dxdt+M_*^2 \intTU \KK  |D^2 u|^2   |\nabla u|^{2s}  \zeta^2  dx dt
\le  C(I_*+J_*),
\eeq
where 
\begin{align*}
I_*&= \chi_*^2M_*^2 \phi \mathcal D_s + T\chi_*^{4(2s+3)}M_*^6 (M_*+1)^{8s+6} \\
&\quad +  \chi_*^{2(4+a)}M_*^2(M_*+1)^4\intTU \KK   |\nabla u|^{2s+2} (\zeta^2 + |\nabla \zeta|^2)   dx dt\\
&\quad  + {\rm sgn}(s)M_*^2 \intTU \KK  |D^2 u|^2 \zeta^2 dxdt,
\end{align*}
and
\begin{align*}
J_*&= T\chi_*^{4(1+a)}M_*^4 (M_*+1)^{4a}\sup|\zeta_t|^2\\
&\quad + \chi_*^{4(1+a\cdot{\rm sgn}(s))}M_*^4 (M_*+1)^{4a\cdot{\rm sgn}(s)}\intTU \KK   |\nabla u|^{2s+2} |\zeta_t|^2 dx dt.
\end{align*}
\end{lemma}
\begin{proof}
Denote
\begin{align*}
\alpha&=\intTU \KK  |\nabla u|^{2s+4} \zeta^2 dxdt,& \gamma&=\intTU |\nabla u|^{2s+2} \zeta|\zeta_t| dx dt,\\
\beta&=\intTU \KK  |D^2 u|^2   |\nabla u|^{2s}  \zeta^2  dx dt,& \beta_0&=\intTU \KK  |D^2 u|^2  \zeta^2  dx dt.
\end{align*}

For $s\ge 0$, by applying \eqref{ibt} to $w(x)=u(x,t)$, $Q(x)=\mathcal Z(x,t)$, $s:=s+1$, and then integrating in $t$ from $0$ to $T$, we have 
\begin{align*}
\alpha
&\le
 C M_*^2 \beta + CM_*^2(1+(M_*\muz+\MZ)^2M_*^2) \intTU \KK |\nabla u|^{2s+2}  (|\nabla \zeta|^2 + \zeta^2 ) dx dt \\
&\quad + C {\rm sgn}(s)M_*^2 \beta_0 + CT(M_*^2\MZ)^{2s+4}(1+M_*^2\MZ)^{2s+2}.
\end{align*} 
Using \eqref{RMm} for upper bounds of $\MZ$ and $\muz$, we have
\begin{align*}
\alpha
&\le C M_*^2 \beta+ C\chi_*^4M_*^2(M_*+1)^4 \intTU \KK |\nabla u|^{2s+2}  (|\nabla \zeta|^2 + \zeta^2 ) dx dt \\
&\quad + C {\rm sgn}(s)M_*^2 \beta_0 + CT\chi_*^{4(2s+3)}M_*^{4(s+2)}(M_*+1)^{4(s+1)}.
\end{align*} 

We estimate $\beta$ by using \eqref{iterate4} and the fact $\widetilde M\le M_*$ to have
\begin{align*}
CM_*^2 \beta
&\le  C \chi_*^2M_*^2\phi \mathcal D_s 
+C\chi_*^{2(4+a)}M_*^6\intTU \KK  |\nabla u|^{2s}    \zeta^2 dx dt \\
&\quad + C\chi_*^{2(4+a)}M_*^2(M_*+1)^2\intTU \KK   |\nabla u|^{2s+2} ( \zeta^2 + |\nabla \zeta|^2)   dx dt
+ C\chi_*^2M_*^2 \gamma.
\end{align*} 
Thus,
\beq \label{albe}
\begin{aligned}
2\alpha+M_*^2 \beta
&\le  C\chi_*^2M_*^2 \phi \mathcal D_s 
+C\chi_*^{2(4+a)}M_*^6\intTU \KK  |\nabla u|^{2s} \zeta^2 dx dt \\
&\quad + C \chi_*^{2(4+a)}M_*^2(M_*+1)^4\intTU \KK   |\nabla u|^{2s+2} (\zeta^2 + |\nabla \zeta|^2)   dx dt\\
&\quad  + C {\rm sgn}(s)M_*^2 \beta_0 + CT\chi_*^{4(2s+3)}M_*^{4(s+2)}(M_*+1)^{4(s+1)}+ C\chi_*^2M_*^2 \gamma.
\end{aligned}
\eeq

For the second term on the right-hand side of \eqref{albe}, the integral is bounded by
\beqs
\intTU \KK  |\nabla u|^{2s} \zeta^2 dx dt
\le \intTU \KK  (|\nabla u|^{2s+2}+1) \zeta^2 dx dt\le \intTU \KK  |\nabla u|^{2s+2}\zeta^2dx dt +CT.
\eeqs

Combining this with the third term on the right-hand side of \eqref{albe} gives
\begin{align*}
&2\alpha+M_*^2 \beta\le  C\chi_*^2M_*^2 \phi \mathcal D_s + C \chi_*^{2(4+a)}M_*^2(M_*+1)^4\intTU \KK   |\nabla u|^{2s+2} (\zeta^2 + |\nabla \zeta|^2)   dx dt\\
&\quad + C {\rm sgn}(s)M_*^2 \beta_0 + CT\chi_*^{4(2s+3)}M_*^{4(s+2)}(M_*+1)^{4(s+1)}+CT\chi_*^{2(4+a)}M_*^6 + C\chi_*^2 M_*^2 \gamma. 
\end{align*}

As far as the two $T$-terms in the last inequality are concerned, the first one has 
 $$M_*^{4(s+2)}(M_*+1)^{4(s+1)}\le M_*^6(M_*+1)^{8s+6},$$
 while the second one has
 $$\chi_*^{2(4+a)}\le \chi_*^{10}\le \chi_*^{4(2s+3)}.$$
 
Therefore,
\beq
2\alpha+M_*^2 \beta\le CI_* + C\chi_*^2 M_*^2 \gamma.\label{ab0}
\eeq

We estimate the last term by using Cauchy's inequality to obtain
\begin{align*}
 C\chi_*^2 M_*^2 \gamma
 &\le  \intTU \KK|\nabla u|^{2s+4} \zeta^2 dx dt
 + C\chi_*^4M_*^4\intTU \KK^{-1}|\nabla u|^{2s}|\zeta_t|^2 dxdt\\
 &=\alpha
 + C\chi_*^4M_*^4\intTU \KK |\nabla u|^{2s} \KK^{-2}|\zeta_t|^2 dxdt.
\end{align*}

In the last integral, we have
\begin{align*}
 |\nabla u|^{2s} \KK^{-2}&\le |\nabla u|^{2s} (1+|\nabla u|+ M_*^2\chi_*^2)^{2a}\\
& \le C(|\nabla u|^{2s} +|\nabla u|^{2s+2a}+ M_*^{4a}\chi_*^{4a}|\nabla u|^{2s})\\
 &\le C[1+|\nabla u|^{2s+2}+ M_*^{4a}\chi_*^{4a}(1+{\rm sgn}(s)|\nabla u|^{2s+2})],
\end{align*}
which can be conveniently rewritten as
\beqs
 |\nabla u|^{2s} \KK^{-2}\le C (M_*+1)^{4a}\chi_*^{4a}+C (M_*+1)^{4a\cdot{\rm sgn}(s)}\chi_*^{4a\cdot{\rm sgn}(s)}|\nabla u|^{2s+2}.
\eeqs

Thus,
\beq\label{ab}
\begin{aligned}
 C\chi_*^2M_*^2 \gamma
 &\le  \alpha
 + C\chi_*^{4(1+a\cdot{\rm sgn}(s))}M_*^4 (M_*+1)^{4a\cdot{\rm sgn}(s)}\intTU \KK |\nabla u|^{2s+2}|\zeta_t|^2 dxdt\\
&\quad + C\chi_*^{4(1+a)}M_*^4 (M_*+1)^{4a}\intTU |\zeta_t|^2 dxdt\le  \alpha + CJ_*.
 \end{aligned}
\eeq 

Combining \eqref{ab0} and \eqref{ab} yields
\beqs
2\alpha+M_*^2 \beta\le CI_*+\alpha+CJ_*,
\text{ which implies } \alpha+M_*^2 \beta\le C(I_*+J_*).
\eeqs
We have proved \eqref{Kug3}.
\end{proof}

As one can see from \eqref{Kug3} that the integral of higher power $2s+4$ of $|\nabla u|$, with the weight $\mathcal K(x,t)$, can be bounded by the corresponding integral of lower power $2s+2$. However, it still involves a second order term, which is the last integral of $I_*$. This term, as it turns out, can be estimated in \eqref{ab4} below.

\subsection{Estimates for the $L_{x,t}^{4-a}$-norm}\label{L4ma}

We start using inequality \eqref{Kug3} with the smallest possible value for $s$, i.e., $s=0$. 
It will result in the $\mathcal K$-weighted $L_{x,t}^2$-estimate, and, consequently, the $L_{x,t}^{4-a}$ estimate for $|\nabla u|$.

\begin{proposition}\label{prop63}
One has
\beq\label{ab4}
\begin{aligned}
&\intTU \KK  |\nabla u|^{4} \zeta^2 dxdt+M_*^2\intTU \KK  |D^2 u|^2 \zeta^2 dxdt\\
&\le C(1+\sup|\nabla \zeta|^2+\sup |\zeta_t|^2)\Big\{ \chi_*^{2(5+a)}M_*^2(M_*+1)^4 \phi( \|\bar u_0\|_{L^2}^2 +\mathcal D_0)\\
&\quad +  \chi_*^{4(4+a)}  T  M_*^4(M_*+1)^8  + \chi_*^{4(3+a)}M_*^2(M_*+1)^4 \Ecal_*\Big\}.
\end{aligned}
\eeq 
\end{proposition}
\begin{proof}
Denote by $I$ the sum on the left-hand side of \eqref{ab4}. It follows \eqref{Kug3} with $s=0$ that 
\begin{align*}
I&\le  C\chi_*^2M_*^2 \phi \mathcal D_0  + CT\cdot \big [\chi_*^{12}M_*^6 (M_*+1)^{6}+\chi_*^{4(1+a)}M_*^4 (M_*+1)^{4a}\sup|\zeta_t|^2\big ] \\
&\quad + C \chi_*^{2(4+a)}M_*^2(M_*+1)^4\intTU \KK   |\nabla u|^{2} (\zeta^2 + |\nabla \zeta|^2)   dx dt\\
&\quad + C \chi_*^{4}M_*^4\intTU \KK   |\nabla u|^{2} |\zeta_t|^2 dx dt.
\end{align*} 
The second term on the right-hand side is bounded by
$$C(1+\sup|\zeta_t|^2)\chi_*^{12} T M_*^4(M_*+1)^8 ,$$
and the sum of the last two terms on the right-hand side is bounded by 
\beqs C(1+\sup |\nabla\zeta|^2+\sup |\zeta_t|^2)  \chi_*^{2(4+a)} M_*^2(M_*+1)^4 
\intTU \KK|\nabla u|^2 dxdt.
\eeqs
Hence, 
\begin{align*}
I&\le   C(1+\sup |\nabla\zeta|^2+\sup |\zeta_t|^2)\Big\{ \chi_*^2M_*^2 \phi \mathcal D_0+\chi_*^{12}  T  M_*^4(M_*+1)^8 \\
&\quad + \chi_*^{2(4+a)}M_*^2(M_*+1)^4  \intTU \KK|\nabla u|^2 dxdt\Big\}.
\end{align*} 

Estimating the last integral by \eqref{gradu2} gives  
\begin{align*}
I&\le   C(1+\sup |\nabla\zeta|^2+\sup |\zeta_t|^2)\Big\{ \chi_*^2M_*^2 \phi \mathcal D_0+\chi_*^{12}  T  M_*^4(M_*+1)^8 \\
&\quad + \chi_*^{2(4+a)}M_*^2(M_*+1)^4  \Big[ \chi_*^2\phi \|\bar u_0\|_{L^2}^2 
 + \chi_*^{2(4+a)} T M_*^2(M_*+1)^2 + \chi_*^{2(2+a)} \Ecal_*\Big]\Big\}.
\end{align*} 

Grouping the like-terms on the right-hand side and using simple estimations yield inequality \eqref{ab4}.
\end{proof}

By selecting the cut-off function $\zeta$ in \eqref{ab4} appropriately, we derive the spatial, as well as the spatial-temporal, interior estimates for $|\nabla u|$. 

\begin{xnotation}
For simplicity, we will write $V\Subset U$ to indicate that $V$ is an open, relatively compact subset of $U$. 
\end{xnotation}

\begin{theorem}\label{thm64}
Let $U'\Subset U$. 

{\rm (i)} One has
\beq\label{ab1}
\begin{aligned}
\intTUp \KK  |\nabla u|^{4} dxdt
&\le C\Big\{ \chi_*^{2(5+a)}M_*^2(M_*+1)^4 \phi( \|\bar u_0\|_{L^2}^2 +\|\nabla u_0\|_{L^2}^2)\\
&\quad +  \chi_*^{4(4+a)}  T  M_*^4(M_*+1)^8  + \chi_*^{4(3+a)}M_*^2(M_*+1)^4 \Ecal_*\Big\}.
\end{aligned}
\eeq 

Consequently,
\beq\label{ab11}
\intTUp \KK  |\nabla u|^{4} dxdt
\le C\chi_*^{4(4+a)} M_*^2(M_*+1)^{10} \Ncal_2.
\eeq 

{\rm (ii)} If $T_0$ is any number in $(0,T)$, then
\beq\label{ab2}
\begin{aligned}
\int_{T_0}^{T}\int_{U'} \KK  |\nabla u|^{4} dxdt
&\le C(1+T_0^{-1})^2\Big\{ \chi_*^{2(5+a)}M_*^2(M_*+1)^4 \phi \|\bar u_0\|_{L^2}^2 \\
&\quad +  \chi_*^{4(4+a)}  T  M_*^4(M_*+1)^8  + \chi_*^{4(3+a)}M_*^2(M_*+1)^4 \Ecal_*\Big\}.
\end{aligned}
\eeq 

Consequently,
\beq\label{ab22}
\int_{T_0}^{T}\int_{U'} \KK  |\nabla u|^{4} dxdt
\le C(1+T_0^{-1})^2\chi_*^{4(4+a)} M_*^2(M_*+1)^{10}\Ncal_*.
\eeq 

{\rm (iii)} If  $s\in[2,4]$, then
\beq \label{ab23}
\intTUp \KK  |\nabla u|^s dxdt \le C\chi_*^{(4+a)s}M_*^{s-2}(M_*+1)^{3s-2}\Ncal_2,
\eeq 
and, for $T_0\in(0,T)$,
\beq \label{ab24}
\int_{T_0}^T\int_{U'} \KK  |\nabla u|^s dxdt \le C(1+T_0^{-1})^{s-2}\chi_*^{(4+a)s}M_*^{s-2}(M_*+1)^{3s-2}\Ncal_*.
\eeq 
\end{theorem}
\begin{proof}
(i) We fix a cut-off function $\zeta=\zeta(x)$ with $\zeta\equiv 1$ on $U'$. 
We have $|\nabla\zeta|\le C$ and $\zeta_t\equiv 0$. Then, by using inequality \eqref{ab4}, we obtain   
\begin{align*}
\intTUp \KK  |\nabla u|^{4}  dxdt
&\le \intTU \KK  |\nabla u|^{4}\zeta^2  dxdt\\
&\le C\Big\{ \chi_*^{2(5+a)}M_*^2(M_*+1)^4 \phi( \|\bar u_0\|_{L^2}^2 +\mathcal D_0)\\
&\quad +  \chi_*^{4(4+a)}  T  M_*^4(M_*+1)^8  + \chi_*^{4(3+a)}M_*^2(M_*+1)^4 \Ecal_*\Big\}.
\end{align*}
This proves \eqref{ab1}.
Now, note on the right-hand side of \eqref{ab1} that  
\beq \label{oo}
\chi_*^{2(5+a)}M_*^2(M_*+1)^4,   \chi_*^{4(4+a)}  M_*^4(M_*+1)^8,  \chi_*^{4(3+a)}M_*^2(M_*+1)^4 \le \chi_*^{4(4+a)}M_*^2(M_*+1)^{10}.
\eeq
Utilizing these estimates, we obtain \eqref{ab11} from \eqref{ab1}.

(ii) We select a different cut-off function $\zeta=\zeta(x,t)$ such that $\zeta=0$ for $0\le t\le T_0/2$, and $\zeta=1$ on $U'\times [T_0,T]$, and its derivatives satisfy  
\beq \label{zprop}
 |\nabla\zeta|\le C \text{ and }0\le \zeta_t \le C T_0^{-1},
\eeq 
where $C>0$ is independent of $T_0,T$.

With this function $\zeta$, it is obvious from \eqref{Ds} that $\mathcal D_0=0$.
Then, by \eqref{ab4}, we have 
\begin{align*}
\int_{T_0}^{T}\int_{U'} \KK  |\nabla u|^{4}  dxdt
&\le \intTU  \KK  |\nabla u|^{4}\zeta^2   dxdt\\
&\le C(1+T_0^{-2})\Big\{ C\chi_*^{2(5+a)}M_*^2(M_*+1)^4 \phi \|\bar u_0\|_{L^2}^2\\
&\quad +  \chi_*^{4(4+a)}  T  M_*^4(M_*+1)^8  + \chi_*^{4(3+a)}M_*^2(M_*+1)^4 \Ecal_*\Big\},
\end{align*}
which gives \eqref{ab2}.
Utilizing \eqref{oo} again for the right-hand side of \eqref{ab2}, we obtain \eqref{ab22}.

(iii) The inequalities \eqref{ab23} and \eqref{ab24} already hold for $s=2$ thanks to \eqref{gradu6a} and \eqref{NNNs}, and for $s=4$ thanks to \eqref{ab1} and \eqref{ab11}. 

Consider $2<s<4$ now. By interpolation inequality \eqref{Lpinter}, we have
 \begin{align*}
 \intTUp \KK  |\nabla u|^s dxdt
 \le \Big( \intTUp \KK  |\nabla u|^2 dxdt\Big)^\frac{4-s}2\Big( \intTUp \KK  |\nabla u|^4 dxdt\Big)^\frac{s-2}2.
 \end{align*}
 
Applying  inequality \eqref{gradu6a}, respectively \eqref{ab11}, to estimate the first, respectively second, integral on the right-hand side, we obtain
\begin{align*}
 \intTUp \KK  |\nabla u|^s dxdt
 &\le C\Big(\chi_*^{2(4+a)}(M_*+1)^4 \Ncal_*\Big)^\frac{4-s}2\Big( \chi_*^{4(4+a)} M_*^2(M_*+1)^{10} \Ncal_2\Big)^\frac{s-2}2,
 \end{align*}
which yields \eqref{ab23}.

Similarly, by \eqref{Lpinter}, \eqref{gradu6a} and \eqref{ab22}, we have
\begin{align*}
 \int_{T_0}^T\int_{U'} \KK  |\nabla u|^s dxdt
& \le \Big( \int_{T_0}^T\int_{U'} \KK  |\nabla u|^2 dxdt\Big)^\frac{4-s}2\Big( \int_{T_0}^T\int_{U'} \KK  |\nabla u|^4 dxdt\Big)^\frac{s-2}2  \\
&\le C \Big(\chi_*^{2(4+a)}(M_*+1)^4 \Ncal_*\Big)^\frac{4-s}2\Big( (1+T_0^{-1})^2\chi_*^{4(4+a)} M_*^2(M_*+1)^{10} \Ncal_*\Big)^\frac{s-2}2,
\end{align*}
which implies \eqref{ab24}.
\end{proof}

The estimates obtained in Theorem \ref{thm64} contain the weight $\KK(x,t)$.
Below, we derive the estimates for the standard Lebesgue $L^{4-a}_{x,t}$-norm (without that weight).

\begin{corollary}\label{cor65}
Let $U'\Subset U$ and $T_0\in(0,T)$.

{\rm (i)} One has
\beq\label{ab31}
\begin{aligned}
\intTUp |\nabla u|^{4-a} dxdt
&\le C\chi_*^{4(4+a)} (M_*+1)^{12} \Ncal_2,
\end{aligned}
\eeq 
and
\beq\label{ab32}
\begin{aligned}
\int_{T_0}^{T}\int_{U'} |\nabla u|^{4-a} dxdt
&\le C(1+T_0^{-1})^2\chi_*^{4(4+a)} (M_*+1)^{12}\Ncal_*.
\end{aligned}
\eeq 

{\rm (ii)} If $s$ is any number in $(2-a,4-a)$, then
\beq \label{ab33}
\intTUp  |\nabla u|^s dxdt
\le C\chi_*^{(4+a)(s+a)} (M_*+1)^{4(s+a-1)}\Ncal_2,
\eeq
and
\beq \label{ab34}
\int_{T_0}^T\int_{U'}  |\nabla u|^s dxdt
\le C(1+T_0^{-1})^{s+a-2}\chi_*^{(4+a)(s+a)} (M_*+1)^{4(s+a-1)}\Ncal_*.
\eeq
 \end{corollary}
\begin{proof}
(i) Applying \eqref{kugs} to $s=4-a$ gives
 \begin{align*}
 |\nabla u|^{4-a}\le C \KK |\nabla u|^4 +C (1+M_*\chi_*)^8.
 \end{align*}
Combining this with  \eqref{ab11}, respectively \eqref{ab22}, we obtain \eqref{ab31}, respectively \eqref{ab32}.

 (ii) Consider $2-a<s<4-a$. By interpolation inequality \eqref{Lpinter} and then using \eqref{gradu6b}, \eqref{NNNs} and \eqref{ab31}, we have
  \begin{align*}
 &\intTUp  |\nabla u|^s dxdt
 \le \Big( \intTUp  |\nabla u|^{2-a} dxdt\Big)^\frac{4-a-s}2\Big( \intTUp |\nabla u|^{4-a} dxdt\Big)^\frac{s-(2-a)}2  \\
 &\le C\Big[\chi_*^{2(4+a)}(M_*+1)^4 \Ncal_*\Big]^\frac{4-a-s}2  \Big[\chi_*^{4(4+a)} (M_*+1)^{12}\Ncal_2\Big]^\frac{s-(2-a)}2 \\
 &\le C\chi_*^{(4+a)(s+a)} (M_*+1)^{4(s+a-1)}\Ncal_2.
 \end{align*}
Thus, we obtain \eqref{ab33}.

Similarly, by \eqref{Lpinter}, \eqref{gradu6b} and \eqref{ab32} we have
 \begin{align*}
 &\int_{T_0}^T\int_{U'}  |\nabla u|^s dxdt
 \le \Big( \int_{T_0}^T\int_{U'}  |\nabla u|^{2-a} dxdt\Big)^\frac{4-a-s}2\Big( \int_{T_0}^T\int_{U'} |\nabla u|^{4-a} dxdt\Big)^\frac{s-(2-a)}2  \\
 &\le C\Big[\chi_*^{2(4+a)}(M_*+1)^4 \Ncal_*\Big]^\frac{4-a-s}2  \Big[(1+T_0^{-1})^2\chi_*^{4(4+a)} (M_*+1)^{12}\Ncal_*\Big]^\frac{s-(2-a)}2 \\
 &\le C(1+T_0^{-1})^{s+a-2}\chi_*^{(4+a)(s+a)} (M_*+1)^{4(s+a-1)}\Ncal_*.
 \end{align*}
Thus, we obtain \eqref{ab34}.
 \end{proof}

 \begin{remark}
The estimate  \eqref{ab34} of the $L^s_{x,t}$-norm of  $\nabla u(x,t)$, for $t>0$, requires, as far as the initial data $u_0$ is concerned, at most the $L^\infty$-norm of $u_0$. Therefore, it shows the (formal) regularization effect of the PDE \eqref{ueq}.
This observation also applies to Corollary \ref{high1} and Theorem \ref{high5} below.
\end{remark}

\subsection{Estimates for higher $L_{x,t}^s$-norms}\label{Lhigher}

In this subsection, we have  estimates for the $L_{x,t}^s$-norms of $\nabla u$ with $s>4-a$.

\begin{lemma}\label{lem67}
Let $s>2$, and $V$ be an open subset of $U$.

{\rm (i)} If $\zeta=\zeta(x)$ with compact support in $V$, then 
\begin{multline}\label{ab66}
\intTU \KK  |\nabla u|^{s+2}\zeta^2 dxdt
\le  C (1+\sup|\nabla \zeta|^2)\\
\cdot \Big\{ \chi_*^{4(s+2+a)}M_*^2 (M_*+1)^{4s+2} \Ncal_s
+ \chi_*^{2(4+a)}M_*^2(M_*+1)^4 \intTV \KK |\nabla u|^s dx dt\Big\}.
\end{multline}

{\rm (ii)}  If $\zeta=\zeta(x,t)$ with $\zeta(x,0)\equiv 0$ and, for each $t\in[0,T]$, the mapping $\zeta(\cdot,t)$ has compact support in $V$,  then
\begin{multline}\label{ab77}
\intTU \KK  |\nabla u|^{s+2} \zeta^2 dxdt
\le  C(1+\sup|\nabla \zeta|^2+\sup|\zeta_t|^2)\\
\cdot \Big\{  \chi_*^{4(s+2+a)}M_*^2 (M_*+1)^{4s+2} \Ncal_*
 + \chi_*^{2(4+a)}M_*^2(M_*+1)^6\intTV \KK   |\nabla u|^s  dx dt\Big\}.
\end{multline}
\end{lemma}
\begin{proof}
Denote $$I=\intTU \KK  |\nabla u|^{2(s+2)}\zeta^2 dxdt \text{ and }J=\intTV \KK |\nabla u|^{2(s+1)}dx dt.$$

(i) Consider $s>0$. We estimate $I$ by \eqref{Kug3}, neglecting the second term on the left-hand side. Note in this case that $\zeta_t=0$ and hence $J_*=0$.  We then use \eqref{ab4} to estimate the last term of $I_*$. The result is 
\begin{align*}
I
&\le  C M_*^2\chi_*^2\phi \mathcal D_s + CT\chi_*^{4(2s+3)}M_*^6 (M_*+1)^{8s+6}  + C\chi_*^{2(4+a)}M_*^2(M_*+1)^4 (1+\sup|\nabla \zeta|^2)  J\\
&\quad +  C(1+\sup|\nabla \zeta|^2)\Big[  \chi_*^{2(5+a)}M_*^2(M_*+1)^4\phi (\|\bar u_0\|_{L^2}^2+\|\nabla u_0\|_{L^2}^2) 
+ T \chi_*^{4(4+a)}M_*^4(M_*+1)^8 \\
&\quad +\chi_*^{4(3+a)}M_*^2(M_*+1)^4 \Ecal_* \Big].
\end{align*} 
For the terms containing the initial data, we estimate
\begin{align*}
M_*^2\chi_*^2\le  \chi_*^{2(5+a)}M_*^2(M_*+1)^4, 
\end{align*}
and for the terms containing $T$, we use
\beq\label{Tcoef}
\chi_*^{4(2s+3)}M_*^6 (M_*+1)^{8s+6},\ \chi_*^{4(4+a)}M_*^4(M_*+1)^8\le \chi_*^{8(s+2)+4a}M_*^4 (M_*+1)^{8(s+1)}.
\eeq
Hence, we obtain
\begin{multline}\label{Kug6}
I\le  C (1+\sup|\nabla \zeta|^2) \Big\{ \chi_*^{2(5+a)}M_*^2(M_*+1)^4 \phi(\|\bar u_0\|_{L^2}^2+ \|\nabla u_0\|_{L^2}^2+\|\nabla u_0\|_{L^{2s+2}}^{2s+2})\\
+ T\chi_*^{8(s+2)+4a}M_*^4 (M_*+1)^{8(s+1)}
+\chi_*^{4(3+a)}M_*^2(M_*+1)^4 \Ecal_*
+ \chi_*^{2(4+a)}M_*^2(M_*+1)^4 J \Big\}.
\end{multline} 

Now, consider $s>2$. By replacing $2s+2$ in \eqref{Kug6} with $s$, noting that
$$I \text{ becomes }\intTU \KK  |\nabla u|^{s+2}\zeta^2 dxdt,\quad J
\text{ becomes }\intTV \KK |\nabla u|^s dx dt,
$$
the power $8(s+2)+4a$ becomes $4(s+2+a)$, and the power $8(s+1)$ becomes $4s$,
we obtain
\begin{multline}\label{ab6}
\intTU \KK  |\nabla u|^{s+2}\zeta^2 dxdt
\le  C (1+\sup|\nabla \zeta|^2)\\
\cdot \Big\{ \chi_*^{2(5+a)}M_*^2(M_*+1)^4 \phi(\|\bar u_0\|_{L^2}^2+ \|\nabla u_0\|_{L^2}^2+\|\nabla u_0\|_{L^s}^s)
+ T\chi_*^{4(s+2+a)}M_*^4 (M_*+1)^{4s}\\ 
+\chi_*^{4(3+a)}M_*^2(M_*+1)^4 \Ecal_*
+ \chi_*^{2(4+a)}M_*^2(M_*+1)^4 \intTV \KK |\nabla u|^s dx dt\Big\}.
\end{multline} 
On the right-hand side of \eqref{ab6}, in order to group the terms $\phi(\|\bar u_0\|_{L^2}^2+ \|\nabla u_0\|_{L^2}^2+\|\nabla u_0\|_{L^s}^s)$, $T$, $\Ecal_*$ together, we estimate their coefficients by 
$$\chi_*^{2(5+a)}M_*^2(M_*+1)^4,\chi_*^{4(s+2+a)}M_*^4 (M_*+1)^{4s},\chi_*^{4(3+a)}M_*^2(M_*+1)^4 
\le \chi_*^{4(s+2+a)}M_*^2 (M_*+1)^{4s+2}.$$
Then inequality \eqref{ab66} follows \eqref{ab6}.

(ii) Consider $s>0$. Note that $\mathcal D_0=\mathcal D_s=0$.  We have from \eqref{Kug3} that 
\begin{align*}
I&\le  C(1+\sup|\nabla \zeta|^2+\sup|\zeta_t|^2)
\Big\{  T\cdot\big[\chi_*^{4(2s+3)}M_*^6 (M_*+1)^{8s+6} + \chi_*^{4(1+a)}M_*^4 (M_*+1)^{4a}\big] \\
&\quad +\big[ \chi_*^{2(4+a)}M_*^2(M_*+1)^4 + \chi_*^{4(1+a)}M_*^4 (M_*+1)^{4a}\big]\cdot J\Big\}
  + M_*^2 \intTU \KK  |D^2 u|^2 \zeta^2 dxdt.
\end{align*} 

We use \eqref{ab4} to estimate the last term $\displaystyle M_*^2 \intTU \KK  |D^2 u|^2 \zeta^2 dxdt$.
For the $T$-term, we use \eqref{Tcoef} again. For the $J$-term we use
\beqs
 \chi_*^{2(4+a)}M_*^2(M_*+1)^4 , \chi_*^{4(1+a)}M_*^4 (M_*+1)^{4a}
 \le  \chi_*^{2(4+a)}M_*^2(M_*+1)^6.
\eeqs
Combining these estimates gives
\begin{align*}
I&\le  C(1+\sup|\nabla \zeta|^2+\sup|\zeta_t|^2)
\Big\{  T\chi_*^{8(s+2)+4a}M_*^4 (M_*+1)^{8(s+1)} 
+ \chi_*^{2(4+a)} M_*^2(M_*+1)^6 J\\
&\quad + \chi_*^{2(5+a)}M_*^2(M_*+1)^4 \phi \|\bar u_0\|_{L^2}^2+ T \chi_*^{4(4+a)}  M_*^4(M_*+1)^8  + \chi_*^{4(3+a)}M_*^2(M_*+1)^4 \Ecal_*\Big\}.
\end{align*} 
Simplifying the right-hand side once more, we obtain 
\beq\label{Kug7}
\begin{aligned}
I
&\le  C(1+\sup|\nabla \zeta|^2+\sup|\zeta_t|^2)
\Big\{  \chi_*^{2(5+a)}M_*^2(M_*+1)^4 \phi \|\bar u_0\|_{L^2}^2  \\
&\quad + T\chi_*^{8(s+2)+4a}M_*^4 (M_*+1)^{8(s+1)} 
 + \chi_*^{4(3+a)}M_*^2(M_*+1)^4 \Ecal_*\\
&\quad  + \chi_*^{2(4+a)}M_*^2(M_*+1)^6 J\Big\}.
\end{aligned}
\eeq 

Same as in the proof of part (i), when $s>2$, replacing $2s+2$ in \eqref{Kug7} with $s$ yields \eqref{ab77}.
\end{proof}

\begin{theorem}\label{high0}
If $U'\Subset U$ and  $s\ge 4$, then 
 \beq\label{ih0}
  \intTUp\KK |\nabla u|^{s}dxdt
  \le C \chi_*^{(4+a)(s+2)}M_*^2 (M_*+1)^{4s}\Ncal_{s-2}.
 \eeq
\end{theorem}
\begin{proof}
(a) When $s=4$ the inequality \eqref{ih0} holds true thanks to the estimate \eqref{ab11}.
Hence we only focus on the case $s>4$.

(b) Consider the case $s=s_*+2m$ with $s_*>2$ and $m\in \N$. 
Let $V$ be an open subset of $U$ such that $U'\Subset V\Subset U$.
We claim that
\beq\label{kug0}
\begin{aligned}
  \intTUp\KK |\nabla u|^{s}dxdt
& \le C\chi_*^{2(4+a)m}M_*^{2m}(M_*+1)^{4m}  \intTV \KK |\nabla u|^{s_*}dxdt\\
&\quad  +C \chi_*^{(4+a)s}M_*^2 (M_*+1)^{4s-6}\Ncal_{s-2}.
\end{aligned}
\eeq

\textit{Proof of \eqref{kug0}.}
Let $\{U_k\}_{k=0}^m$ be a family of smooth, open subsets of $U$ such that  
 \beq\label{Usets}
 U'\subset U_m\Subset U_{m-1}\Subset U_{m-2}\Subset \ldots \Subset U_{1} \Subset U_0 \subset V\Subset U.
 \eeq 

 Denote
$\displaystyle y_k=\intTUk \KK |\nabla u|^{s_*+2k}dx dt$ for $0\le k\le m$.

Let $k\in\{0,1,2,\ldots, m-1\}$. Choose $\zeta=\zeta_k(x)$, a $C^2$ cut-off function 
which is equal to $1$ on $U_{k+1}$ and has compact support in $U_k$. 
Applying \eqref{ab66} to $s:=s_*+2k$, we have  
\beq\label{yk2} y_{k+1}\le A y_k+B, 
\eeq
where $A=C_k\chi_*^{2(4+a)} M_*^2(M_*+1)^4$
and $\displaystyle B= C_k\chi_*^{4(s_*+2k+2+a)}M_*^2(M_*+1)^{4(s_*+2k)+2}\widehat B$,
with
\beqs
\widehat B
=\phi(\|\bar u_0\|_{L^2}^2+ \|\nabla u_0\|_{L^2}^2+\|\nabla u_0\|_{L_{s_*+2k}}^{s_*+2k})+T+\Ecal_*
\eeqs
for some $C_k>0$ independent of $T$.
Note that
\begin{align*}
\|\nabla u_0\|_{L_{s_*+2k}}^{s_*+2k}
=\int_U |\nabla u_0|^{s_*+2k}dx\le \int_U \big[|\nabla u_0|^2+|\nabla u_0|^{s_*+2(m-1)}\big]dx.   
\end{align*}
Hence,
\beqs
\widehat B
\le \phi(\|\bar u_0\|_{L^2}^2+ 2\|\nabla u_0\|_{L^2}^2+\|\nabla u_0\|_{L^{s_*+2(m-1)}}^{s_*+2(m-1)})+T+\Ecal_*
\le 2\Ncal_{s-2}.
\eeqs

Let $C_*=2\max\{C_k:k=1,2,\ldots,m-1\}$.
Hence,
\beq\label{yk1} y_{k+1}\le A_* y_k+B_k, 
\eeq
where 
\begin{align*}
A_*=C_*\chi_*^{2(4+a)} M_*^2(M_*+1)^4,\quad 
B_k= C_*\chi_*^{4(s_*+2k+2+a)}M_*^2(M_*+1)^{4(s_*+2k)+2}\Ncal_{s-2}
= B_* S^k,
\end{align*}
with 
$ S=\chi_*^8(M_*+1)^8$ and $B_*=A_*\chi_*^{4s_*+2a}(M_*+1)^{4s_*-2}\Ncal_{s-2}$.

Iterating \eqref{yk1}, we obtain
\begin{align*}
 y_{k+1}
 &\le A_*( A_*y_{k-1}+B_{k-1})+B_{k}
 = A_*^2y_{k-1}+A_*B_{k-1}+B_{k}\\
 &\le A_*^3y_{k-2}+A_*^2B_{k-2}+A_*B_{k-1}+B_{k}\\
 &\le \cdots\le A_*^{k+1}y_{0}+\sum_{j=0}^{k-1} A_*^{k-j} B_j + B_{k}. 
\end{align*}

Letting $k=m-1$, we then have
\beq\label{yk0}
y_m \le A_*^m y_0 + \sum_{j=0}^{m-2} A_*^{m-1-j}B_j+B_{m-1} . 
\eeq

Dealing with the middle sum on the right-hand side of \eqref{yk0}, elementary calculations show, for $0\le j\le m-2$, that
\begin{align*}
A_*^{m-j-1}B_j
&=A_*^{m-j} \chi_*^{4s_*+2a+8j}(M_*+1)^{4s_*-2+8j}\Ncal_{s-2}\\
&= C \chi_*^{4(s_*+2m)+2a(m-j+1)} M_*^{2(m-j)} (M_*+1)^{4(s_*+2m)-4(m-j)-2}\Ncal_{s-2}\\
&\le C \chi_*^{4s+2a(m+1)}\cdot [ M_*^2 (M_*+1)^{2(m-j-1)} ]\cdot (M_*+1)^{4s-4(m-j)-2}\Ncal_{s-2}\\
&= C \chi_*^{4s+2a(m+1)} M_*^2  (M_*+1)^{4s-2(m-j)-4}\Ncal_{s-2}.
\end{align*}
Note that $m-j\ge 2$ and 
\beq\label{sm}
s>2+2m.
\eeq
Then we have
\beq\label{abm}
A_*^{m-1-j}B_j
\le C \chi_*^{(4+a)s} M_*^2(M_*+1)^{4s-8}\Ncal_{s-2}.
\eeq

For the last term in \eqref{yk0}, one has
\begin{align}\label{bm}
 B_{m-1}
 &=C\chi_*^{4(s_*+2m+a)}M_*^2(M_*+1)^{4(s_*+2m)-6}\Ncal_{s-2}\notag\\
 &=C\chi_*^{4(s+a)}M_*^2(M_*+1)^{4s-6}\Ncal_{s-2}.
\end{align}

Then combining \eqref{yk0} with \eqref{abm} and \eqref{bm} gives
\begin{align*}
y_m&\le C\chi_*^{2(4+a)m}M_*^{2m}(M_*+1)^{4m}y_0+C (m-1)\chi_*^{(4+a)s} M_*^2(M_*+1)^{4s-8}\Ncal_{s-2}\\
&\quad +C\chi_*^{4(s+a)}M_*^2(M_*+1)^{4s-6}\Ncal_{s-2}\\
&\le C\chi_*^{2(4+a)m}M_*^{2m}(M_*+1)^{4m}y_0+C \chi_*^{(4+a)s}M_*^2 (M_*+1)^{4s-6}\Ncal_{s-2}.
\end{align*}
Therefore, estimate \eqref{kug0} follows.

(c) Consider the general case $s>4$ now. Then there exist  $s_*\in(2,4]$ and integer $m\ge 1$ such that $s=s_*+2m$.
We apply estimate \eqref{kug0} using the relation \eqref{sm}, and  have
\begin{align*}
  \intTUp\KK |\nabla u|^{s}dxdt
& \le C\chi_*^{(4+a)(s-2)}M_*^{2m}(M_*+1)^{2(s-2)}  \intTV \KK |\nabla u|^{s_*}dxdt\\
&\quad  +C \chi_*^{(4+a)s}M_*^2 (M_*+1)^{4s-6}\Ncal_{s-2}.
\end{align*}
Using \eqref{sm} again,
$$M_*^{2m}\le M_*^2(M_*+1)^{2m-2}\le M_*^2(M_*+1)^{s-4}.$$
Then
\begin{align*}
  \intTUp\KK |\nabla u|^{s}dxdt
& \le C\chi_*^{(4+a)(s-2)}M_*^2(M_*+1)^{3s-8}  \intTV \KK |\nabla u|^{s_*}dxdt\\
&\quad  +C \chi_*^{(4+a)s}M_*^2 (M_*+1)^{4s-6}\Ncal_{s-2}.
\end{align*}

Note, by Young's inequality and applying \eqref{ab11} to $U':=V$, that
\begin{align*}
&\intTV \KK |\nabla u|^{s_*}dxdt
\le C(T+\intTV \KK |\nabla u|^{4}dxdt) \\
&\le  C \chi_*^{4(4+a)}(M_*+1)^{12}\big[ \phi (\|\bar u_0\|_{L^2}^2+\|\nabla u_0\|_{L^2}^2)  +T+\Ecal_*\big].
\end{align*}

Due to \eqref{NNNs} we can conclude that
\beqs
  \intTUp\KK |\nabla u|^{s}dxdt
\le C\chi_*^{(4+a)(s+2)}M_*^2(M_*+1)^{3s+4} \Ncal_2  +C \chi_*^{(4+a)s}M_*^2 (M_*+1)^{4s-6}\Ncal_{s-2},
\eeqs
and we obtain \eqref{ih0}.
\end{proof}

\begin{proposition}\label{high2}
If $U'\Subset V\Subset U$, and $s=s_*+2m$ with $s_*>2$ and  $m\in \N$, then
\beq\label{kug3}
\begin{aligned}
  \int_{T_0}^{T}\int_{U'}\KK |\nabla u|^{s}dxdt
&  \le C(1+t_0^{-1})^{2m}\chi_*^{2(4+a)m}M_*^{2m}(M_*+1)^{6m}  \int_{T_0-t_0}^{T}\int_{V}\KK |\nabla u|^{s_*}dxdt\\
&\quad   +C(1+t_0^{-1})^{2m}\chi_*^{(4+a)s} M_*^2(M_*+1)^{4s-6}\Ncal_*.
\end{aligned}
   \eeq
for any numbers $T_0$ and $t_0$ such that $0<t_0<T_0<T$. 
\end{proposition}
\begin{proof}
Let $\{U_k\}_{k=0}^m$ be as in \eqref{Usets}.
Let $\tau_0=T_0-t_0<\tau_1<\tau_2<\ldots<\tau_m=T_0$ be evenly paced. 
Define 
$\displaystyle y_k=\int_{\tau_k}^{T}\int_{U_k} \KK |\nabla u|^{s_*+2k} dxdt$ for $0\le k\le m$.

Given $k\in\{0,1,2,\ldots, m-1$\}. Let $\zeta_k(x,t)$ be a smooth cut-off function 
 which is equal to one on $U_{k+1}\times [\tau_{k+1},T]$,  has compact support in $U_k\times[\tau_k,T]$, and satisfies 
$$|\nabla \zeta_k|\le C'_k,\quad 0\le \zeta_{k,t}\le \frac2{\tau_{k+1}-\tau_k}=\frac2{mt_0},$$
where $C'_k>0$ is independent of $T,T_0,t_0$.
Then using $s:=s_*+2k$ and $\zeta=\zeta_k$ in \eqref{ab77}, we have  the same relation \eqref{yk2}, with the constants defined by 
 \begin{align*} 
 A&=C_k(1+t_0^{-2})\chi_*^{2(4+a)}M_*^2(M_*+1)^6,\\
 B&=C_k(1+t_0^{-2})\chi_*^{4(s_*+2k+2+a)}M_*^2(M_*+1)^{4(s_*+2k)+2}\Ncal_*,
\end{align*}
for some positive constant $C_k$ independent of $T,T_0,t_0$.

Set $C_*=\max\{C_k:k=0,1,\ldots,m-1\}$, we obtain \eqref{yk1} where 
\begin{align*}
A_*&=C_*(1+t_0^{-1})^2\chi_*^{2(4+a)}M_*^2(M_*+1)^6,\\
B_k&=C_*(1+t_0^{-1})^2\chi_*^{4(s_*+2k+2+a)}M_*^2(M_*+1)^{4(s_*+2k)+2}\Ncal_*
= B_* S^k
\end{align*}
with the same $S=\chi_*^8(M_*+1)^8$, but 
\begin{align*}
 B_*&=C_*(1+t_0^{-1})^2\chi_*^{4(s_*+2+a)}M_*^2(M_*+1)^{4s_*+2}\Ncal_*=A_* \chi_*^{4s_*+2a}(M_*+1)^{4s_*-4}\Ncal_*.
\end{align*}
Then we obtain \eqref{yk0} by iteration again.
For $0\le j\le m-2$,
\begin{align*}
A_*^{m-j-1}B_j
&=A_*^{m-j} \chi_*^{4s_*+2a+8j}(M_*+1)^{4s_*+8j-4}\Ncal_*\\
&\le C(1+t_0^{-1})^{2(m-j)} \chi_*^{4(s_*+2m)+2a(m-j+1)} M_*^{2(m-j)} (M_*+1)^{4s_*+6m+2j-4}\Ncal_*.
\end{align*}
Simply estimating $M_*^{2(m-j)}\le  M_*^2 (M_*+1)^{2(m-j-1)}$, we then have
\begin{align*}
A_*^{m-j-1}B_j
&\le  C(1+t_0^{-1})^{2(m-j)} \chi_*^{4(s_*+2m)+2a(m+1)} M_*^2(M_*+1)^{4s_*+8m-6}\Ncal_*\\
&= C(1+t_0^{-1})^{2m} \chi_*^{(4+a)s} M_*^2(M_*+1)^{4s -6}\Ncal_*.
\end{align*}
Also,
\beqs
B_{m-1}=C(1+t_0^{-1})^2\chi_*^{4(s+a)}M_*^2(M_*+1)^{4s-6}\Ncal_*.
\eeqs
Thus, we have from \eqref{yk0} that
\begin{align*}
y_m&\le C(1+t_0^{-1})^{2m}\chi_*^{2(4+a)m}M_*^{2m}(M_*+1)^{6m}y_0\\
&\quad +C (m-1)(1+t_0^{-1})^{2m} \chi_*^{(4+a)s} M_*^2(M_*+1)^{4s -6}\Ncal_*\\
&\quad +C(1+t_0^{-1})^2\chi_*^{4(s+a)}M_*^2(M_*+1)^{4s-6}\Ncal_*.
\end{align*}
Hence, we obtain \eqref{kug3}.
\end{proof}

\begin{theorem}\label{high3}
If $U'\Subset U$ and $s>4$, then one has, for any $T_0\in (0,T)$,  that
 \beq \label{kug4}
  \int_{T_0}^{T}\int_{U'}\KK |\nabla u|^{s}dxdt
  \le C (1+T_0^{-1})^s \chi_*^{(4+a)(s+2)}M_*^2 (M_*+1)^{4s+2}\Ncal_*.
 \eeq 
\end{theorem}
\begin{proof}
There exist $2<s_*\le 4$ and integer $m\ge 1$ such that $s=s_*+2m$. Let $t_0:=T_0/2$, and $V$ be a set  with $U'\Subset V\Subset U$.  
Applying \eqref{kug3}, we have
\beq\label{kug3b}
\begin{aligned}
  \int_{T_0}^{T}\int_{U'}\KK |\nabla u|^{s}dxdt
&  \le C(1+T_0^{-1})^{2m}\chi_*^{2(4+a)m}M_*^{2m}(M_*+1)^{6m}  y_*\\
&\quad   +C(1+T_0^{-1})^{2m}\chi_*^{(4+a)s} M_*^2(M_*+1)^{4s-6}\Ncal_*.
\end{aligned}
   \eeq
where  
$\displaystyle y_*=\int_{T_0/2}^{T}\int_V\KK |\nabla u|^{s_*}dxdt$.
By Young's inequality and \eqref{ab22} applied to $U':=V$, we have 
\beqs y_*\le C\Big(T+\int_{T_0/2}^{T}\int_V\KK |\nabla u|^{4}dxdt\Big) 
\le  C(1+T_0^{-1})^2 \chi_*^{4(4+a)}(M_*+1)^{12}\Ncal_*. 
\eeqs 
Then
\begin{align*}
 C&(1+T_0^{-1})^{2m}\chi_*^{2(4+a)m}M_*^{2m}(M_*+1)^{6m}y_*
\le C(1+T_0^{-1})^{s-2}\chi_*^{(4+a)(s-2)}M_*^{2}(M_*+1)^{8m-2}y_*
\\
&\le C(1+T_0^{-1})^{s}\chi_*^{(4+a)(s+2)} M_*^2(M_*+1)^{4s+2}\Ncal_*.
\end{align*}
Combining this with \eqref{kug3b} gives \eqref{kug4}.
\end{proof}

\begin{corollary}\label{high1} Let $U'\Subset U$ and $s> 4-a$.
Then
 \beq\label{ih1}
  \intTUp |\nabla u|^{s}dxdt
    \le C \chi_*^{(4+a)(s+a+2)}(M_*+1)^{4(s+a+1/2)}\Ncal_{s+a-2}.
 \eeq 
Moreover, it holds, for any number $T_0\in(0,T)$, that
 \beq \label{ih2}
  \int_{T_0}^{T}\int_{U'}|\nabla u|^{s}dxdt
  \le C (1+T_0^{-1})^{s+a} \chi_*^{(4+a)(s+a+2)} (M_*+1)^{4(s+a+1)}\Ncal_*.
 \eeq 
 \end{corollary}
\begin{proof}
Using \eqref{kugs} and applying \eqref{ih0} with $s$ being substituted by $s+a$, we have
 \begin{align*}
   \intTUp|\nabla u|^{s}dxdt
   &\le C\intTUp (\KK|\nabla u|^{s+a} + (\chi_*(M_*+1))^{2(s+a)})dxdt\\
  &\le C \chi_*^{(4+a)(s+a+2)}M_*^2 (M_*+1)^{4(s+a)}\Ncal_{s+a-2} + CT\chi_*^{2(s+a)}(M_*+1)^{2(s+a)}.
 \end{align*}
 Note that $T\le  \Ncal_{s+a-2}$. Then \eqref{ih1} follows.
 Similarly, using \eqref{kug4}, instead of \eqref{ih0}, we obtain \eqref{ih2}.
\end{proof}

\section{Gradient estimates (III)}\label{maxintime}
This section is focused on the estimates for the $L_t^\infty L_x^s$-norms of $\nabla u$.
For $s\ge 2$, replacing $s$ in \eqref{irat0} with $s/2-1$ gives
\beq\label{iterate2}
\begin{aligned}
I
&\eqdef \phi \sup_{t\in[0,T]}\int_U |\nabla u(x,t)|^s \zeta^2(x,t) dx\\
&\le \phi \int_U |\nabla u_0(x)|^s \zeta^2(x,0) dx
 +C\chi_*^{2(3+a)}M_*^4\intTU \KK  |\nabla u|^{s-2}  \zeta^2 dx dt \\
&\quad + C \chi_*^{2(3+a)}(M_*+1)^2\intTU \KK   |\nabla u|^s (\zeta^2 + |\nabla \zeta|^2)   dx dt + C\intTU |\nabla u|^s \zeta|\zeta_t| dx dt.
\end{aligned}
\eeq

\begin{theorem}\label{high4}
If $U'\Subset U$, then one has, for all $t\in[0,T]$, that
\begin{align}\label{pwtall}
 \phi \int_{U'} |\nabla u(x,t)|^s  dx
&\le \phi \int_U |\nabla u_0(x)|^s dx \notag \\
&\quad +C\begin{cases}
  \chi_*^{4(4+a)}(M_*+1)^{6}\Ncal_0 &\text{ if }s=2,\\ 
  \chi_*^{(s+2)(4+a)}M_*^{s-2}(M_*+1)^{3s+2}\Ncal_2 &\text{ if }2<s\le 4,\\
  \chi_*^{(s+4)(4+a)}M_*^2 (M_*+1)^{4(s+1)}\Ncal_{s-2} &\text{ if }s>4.
         \end{cases}
\end{align}
\end{theorem}
\begin{proof}
Denote $\displaystyle J=\phi \sup_{t\in[0,T]}\int_{U'} |\nabla u(x,t)|^s  dx.$
Choose $\zeta$ to be the same function $\zeta(x)$ as in the proof of Theorem \ref{thm64}(i).
Then we have the relation
\beq\label{JI}
J\le I.
\eeq
We then bound $I$ by using inequality \eqref{iterate2}, noticing that the last integral of this inequality vanishes,
and the integrand of the second term on its right-hand side can be bounded by
\begin{align*}
\KK  |\nabla u|^{s-2} \le \KK  (1 + |\nabla u|^s)\le 1+\KK |\nabla u|^s .
\end{align*}
After this, combining the two constants for the integrals involving $\KK |\nabla u|^{s} $, we obtain
\beq\label{start7} 
\begin{aligned}
J&\le \phi \int_U |\nabla u_0(x)|^s dx + C\chi_*^{2(3+a)}M_*^4 T\\
&\quad +C\chi_*^{2(3+a)}(M_*+1)^4 \intTU \KK |\nabla u|^s (\zeta^2 + |\nabla \zeta|^2)   dx dt.
\end{aligned}
\eeq 

Consider $s=2$. Using \eqref{gradu4} to estimate the last integral in \eqref{start7}, we obtain
\begin{align*}
J&\le \phi \int_U |\nabla u_0(x)|^s dx + C\chi_*^{2(3+a)}M_*^4 T\\
&\quad +C\chi_*^{2(3+a)}(M_*+1)^4 \cdot \chi_*^{2(4+a)}(M_*+1)^2 \Ncal_0.
\end{align*}
Making a generous bound $2(3+a)<2(4+a)$ for the first two exponents of $\chi_*$ above, we obtain the first estimate in \eqref{pwtall}.  

Consider $2< s\le 4$. Using \eqref{ab23} to estimate the last integral in \eqref{start7}, we obtain
\begin{align*}
J&\le \phi \int_U |\nabla u_0(x)|^s dx + C\chi_*^{2(3+a)}M_*^4 T\\
&\quad +C\chi_*^{2(3+a)}(M_*+1)^4 \cdot \chi_*^{(4+a)s}M_*^{s-2}(M_*+1)^{3s-2}\Ncal_2.
\end{align*}
Then the second  estimate in \eqref{pwtall} follows.

Consider $s> 4$. Using \eqref{ih0} to estimate the last integral in \eqref{start7}, we have
\begin{align*}
J&\le \phi \int_U |\nabla u_0(x)|^s dx + C\chi_*^{2(3+a)}M_*^4 T\\
&\quad +C\chi_*^{2(3+a)}(M_*+1)^4 \cdot \chi_*^{(4+a)( s+2)}M_*^2 (M_*+1)^{4s}
\Ncal_{s-2}.
\end{align*}
With simple manipulations, we obtain from this the third estimate in \eqref{pwtall}.
\end{proof}

\begin{theorem}\label{high5}
Let $U'\Subset U$ and $T_0\in(0,T)$. Then it holds, for all $t\in[T_0,T]$, that
\begin{align}\label{pwtnew}
&\phi\int_{U'} |\nabla u(x,t)|^{s}dx
\le C \notag \\
& \cdot
\begin{cases}
\chi_*^{(4+a)^2}(1+T_0^{-1})^{1+a}(M_*+1)^{2(3+a)}\big\{ M_*^{a} (M_*+1)^{2+a} \Ncal_*+\Ncal_0\big\}
&\text{ if }s=2,\\ 
\chi_*^{(4+a)(s+a+2)} (1+T_0^{-1})^{s+a-1} M_*^{s-2} (M_*+1)^{3s+4a+2} \Ncal_*
&\text{ if }2< s\le 4-a,\\
\chi_*^{(4+a)(s+a+4)} (1+T_0^{-1})^{s+a+1} M_*^{s-2} (M_*+1)^{3s+4a+10} \Ncal_*
&\text{ if }4-a<s\le 4,\\
\chi_*^{(4+a)(s+a+4)}(1+T_0^{-1})^{s+a+1}M_*^2 (M_*+1)^{4s+4a+6} \Ncal_*
&\text{ if }s>4.
\end{cases}
\end{align}

Consequently, one has, for all $s\ge 2$ and $t\in[T_0,T]$, that
 \beq\label{pwt6}
\phi\int_{U'} |\nabla u(x,t)|^{s}dx
\le  C  \chi_*^{(4+a)(s+a+4)}(1+T_0^{-1})^{s+a+1}(M_*+1)^{4(s+a+2)} \Ncal_*.
\eeq
\end{theorem}
\begin{proof}
Choose $\zeta(x,t)$ to be the cut-off function in the proof of Theorem \ref{thm64}(ii) which satisfies additionally that 
$\zeta$ has compact support in $V\times [T_0/2,T]$, where $U'\Subset V\Subset U$.

Let $J$ be the same as in Theorem \ref{high4}. Again, we have \eqref{JI}, and  use \eqref{iterate2} to estimate $I$. 
Note, on the right-hand side of \eqref{iterate2}, that 
\beqs
\KK  |\nabla u|^{s-2}\le 1+|\nabla u|^s,\quad  \KK  |\nabla u|^s\le  |\nabla u|^s. 
\eeqs
Utilizing these properties as well as \eqref{zprop}, we have from \eqref{JI} and \eqref{iterate2} that
\beq\label{JJ}
J\le C\chi_*^{2(3+a)}M_*^4 T
 + C \chi_*^{2(3+a)}(M_*+1)^4(1+T_0^{-1}) \int_{T_0/2}^T\int_V  |\nabla u|^s  dx dt.
\eeq
Estimate the last integral in \eqref{JJ},
\begin{align*}
&\int_{T_0/2}^{T}\int_V |\nabla u|^s dx dt
 = \int_{T_0/2}^{T}\int_V \KK|\nabla u|^s \KK^{-1}dx dt\\
  &\le C\int_{T_0/2}^{T}\int_V \KK(|\nabla u|^{s+a}+|\nabla u|^s\chi_*^{2a}(M_*+1)^{2a})dxdt\\
 &= C\int_{T_0/2}^{T}\int_V \KK|\nabla u|^{s+a}dxdt+C\chi_*^{2a}(M_*+1)^{2a} \int_{T_0/2}^{T}\int_V \KK|\nabla u|^s dxdt.
\end{align*}

Denote by $I_1$ and $I_2$ the last two double integrals. We estimate them, in calculations below, by using inequalities \eqref{ab24} and \eqref{kug4} with $T_0:=T_0/2$ and $U':=V$.

Case $s=2$. Applying \eqref{ab24} to $s:=2+a\in(2,4)$ to bound $I_1$, and applying \eqref{gradu4} to $s:=2$ to bound $I_2$ give
\begin{align*}
J
&\le C\chi_*^{2(3+a)}M_*^4 T
  + C \chi_*^{2(3+a)}(M_*+1)^4(1+T_0^{-1})\\
  &\quad \cdot\Big\{  (1+T_0^{-1})^a \chi_*^{(4+a)(2+a)}M_*^a(M_*+1)^{4+3a}\Ncal_*  + \chi_*^{2a}(M_*+1)^{2a} \cdot \chi_*^{2(4+a)}(M_*+1)^2 \Ncal_0\Big\}\\
&\le C\chi_*^{2(3+a)}M_*^4 T
  + C \chi_*^{(4+a)^2}(1+T_0^{-1})^{1+a}M_*^{a} (M_*+1)^{8+3a} \Ncal_*\\
&\quad    + C \chi_*^{14+6a}(1+T_0^{-1})(M_*+1)^{6+2a} \Ncal_0.
\end{align*}
We obtain the first estimate in \eqref{pwtnew}.

Case $2<s\le 4-a$. Estimating $I_1$ by \eqref{ab24} applied to $s:=s+a$, and estimating $I_2$ by \eqref{ab24}, we have 
\begin{align*}
J&\le C\chi_*^{2(3+a)}M_*^4 T
  + C \chi_*^{2(3+a)}(M_*+1)^4(1+T_0^{-1})\\
  &\quad \cdot\Big\{  (1+T_0^{-1})^{s+a-2}\chi_*^{(4+a)(s+a)}M_*^{s+a-2}(M_*+1)^{3(s+a)-2}\Ncal_* \\
&\quad + \chi_*^{2a}(M_*+1)^{2a} \cdot (1+T_0^{-1})^{s-2}\chi_*^{(4+a)s}M_*^{s-2}(M_*+1)^{3s-2}\Ncal_*\Big\}\\
&\le C\chi_*^{2(3+a)}M_*^4 T
  + C \chi_*^{(4+a)(s+a+2)}(1+T_0^{-1})^{s+a-1} M_*^{s-2} (M_*+1)^{3s+4a+2}\Ncal_*.
\end{align*}
We obtain the second estimate in \eqref{pwtnew}.
  
Case $4-a<s\le 4$. Estimating $I_1$ by \eqref{kug4} applied to $s:=s+a$, and estimating $I_2$ by  \eqref{ab24} yield
\begin{align*}
J&\le C\chi_*^{2(3+a)}M_*^4 T
  + C \chi_*^{2(3+a)}(M_*+1)^4(1+T_0^{-1})\\
  &\quad \cdot\Big\{  (1+T_0^{-1})^{s+a} \chi_*^{(4+a)(s+a+2)}M_*^2 (M_*+1)^{4(s+a)+2}\Ncal_* \\
&\quad + \chi_*^{2a}(M_*+1)^{2a} \cdot (1+T_0^{-1})^{s-2}\chi_*^{(4+a)s}M_*^{s-2}(M_*+1)^{3s-2}\Ncal_*\Big\}\\
&\le C\chi_*^{2(3+a)}M_*^4 T
  + C \chi_*^{(4+a)(s+a+4)}(1+T_0^{-1})^{s+a+1}M_*^{s-2} (M_*+1)^{3s+4a+10} \Ncal_*.
\end{align*}
We obtain the third estimate in \eqref{pwtnew}.

Case $s>4$. Estimating $I_1$ by \eqref{kug4} for $s:=s+a$, and estimating $I_2$ by \eqref{kug4} result in 
\begin{align*}
J&\le C\chi_*^{2(3+a)}M_*^4 T
  + C \chi_*^{2(3+a)}(M_*+1)^4(1+T_0^{-1})\\
  &\quad \cdot\Big\{  (1+T_0^{-1})^{s+a} \chi_*^{(4+a)(s+a+2)}M_*^2 (M_*+1)^{4(s+a)+2}\Ncal_* \\
&\quad + \chi_*^{2a}(M_*+1)^{2a} \cdot (1+T_0^{-1})^s \chi_*^{(4+a)(s+2)}M_*^2 (M_*+1)^{4s+2}\Ncal_*\Big\}\\
&\le C\chi_*^{2(3+a)}M_*^4 T
  + C \chi_*^{(4+a)(s+a+4)}(1+T_0^{-1})^{s+a+1}M_*^2 (M_*+1)^{4s+4a+6} \Ncal_*.
\end{align*}
We obtain the fourth estimate in \eqref{pwtnew}.

Finally, one can easily unify the estimates in \eqref{pwtnew} for all $s>2$ with \eqref{pwt6}. This can also be done for the case $s=2$ by comparing $\Ncal_0$ with $\Ncal_*$ using the last relation in \eqref{NNN0}.
\end{proof}

\begin{remark}\label{smlrmk2}
Similar to Remark \ref{smlrmk1}, when $u$, $u_0$, $\bar u_0$ are small in necessary norms, and $\Ecal_*$ is small, then $M_*$ and $\Ncal_0$ are small, which make the the right-hand sides of \eqref{pwtall} and \eqref{pwtnew} to be small.
\end{remark}

\bigskip
\noindent\textbf{\large Acknowledgments.} The authors would like to thank Dat Cao, Akif Ibragimov and Tuoc Phan for very helpful discussions.

\bigskip
\noindent\textbf{\large Data availability.} 
The data that support the findings of this study are available from the corresponding author
upon reasonable request. The data that supports the findings of this study are available within the article.


\def\cprime{$'$}

\end{document}